\documentclass[10pt,a4paper]{amsart}
\usepackage{amssymb,mathrsfs,mathtools,comment,multirow,longtable,mfpic,caption,booktabs,url}
\textheight=\paperheight
\advance\textheight-1in
\advance\textheight-\topmargin
\advance\textheight-\headheight
\advance\textheight-\headsep
\advance\textheight-\footskip
\advance\textheight-\topmargin
\advance\textheight-1in

\evensidemargin=\oddsidemargin
\textwidth=\paperwidth
\advance\textwidth-1in
\advance\textwidth-\oddsidemargin
\advance\textwidth-\marginparsep
\advance\textwidth-\marginparwidth
\advance\textwidth-\oddsidemargin
\advance\textwidth-1in

\newcommand{\Z}{{\mathbb Z}}
\newcommand{\Q}{{\mathbb Q}}

\newcommand{\C}{{\mathbb C}}
\newcommand{\F}{{\mathbb F}}

\newtheorem{thrm}{Theorem}
\newtheorem{prop}{Proposition}
\newtheorem{lemma}{Lemma}
\newtheorem{defi}{Definition}

\newcommand{\ehat}{\widehat{e\vphantom{d}}}
\newcommand{\dhat}{\widehat{d}}
\allowdisplaybreaks
\date{\today}
\begin{document}
\title{On the power of the discriminant of a univariate polynomial as a certain determinant in positive characteristic}
\author{Akira Kurihara}
\address{Professor Emeritus, Japan Women's University, Dept. of Mathematical and Physical Sciences, 2-8-1 Mejirodai, Bunkyo-ku, 112-8681 Tokyo, Japan}
\email{kurihara@fc.jwu.ac.jp}
\begin{abstract}
Let $p$ be a prime.
Suppose that integers $r$, $e$, $d$ such that $r \ge 2$, $e \ge 0$, $0 \le d \le p$ are given.
Let $f(x)=s_0 x^r + s_1 x^{r-1} + \cdots + s_r$ be a generic polynomial of degree $r$ in characteristic $p$.
We put $f(x)^e=\sum_{i \ge 0} c_i x^i$.
We define a $d\times d$ matrix $M_d(f(x)^e)$ by $M_d(f(x)^e) = ( c_{i p + j - d -1})_{1 \le i,\, j \le d}$.
In this paper, we shall be concerned with the divisibility of $\det M_d(f(x)^e)$ by powers of the discriminant $\Delta(f(x))$ of $f(x)$.
First, assuming $s_0=1$, we study the condition under which $\det M_d(f(x)^e)$ is a positive power of $\Delta(f(x))$ multiplied by a non-zero constant in $\F_p$.
Second, for such matrices when $d=r-1$, we present a formula for $M_d(f(x)^e)^{-1} M_d(f(x)^{e+1})$ involving the B\'ezout matrix of $f'(x)$ and $f(x)-\frac{1}{r} x f'(x)$.
Finally, we present two similar experimental equalities, the first of which involves the determinant $\det M_d(f(x)^e)$.
\end{abstract}
\maketitle
\tableofcontents
%
%
%
%
%
\section{Introduction}
%
%
\subsection{A history of discriminant as determinant}
Suppose two polynomials
$\displaystyle F(x)=\sum_{0 \le i \le m} a_{m-i}x^i=a_0 \prod_{1 \le i \le m} (x-\xi_i)$
and 
$\displaystyle G(x)=\sum_{0 \le i \le n} b_{n-i}x^i=b_0 \prod_{1 \le i \le n} (x-\eta_i)$
of degree $m$ and $n$ are given.
The Sylvester matrix ${\rm Syl}(F,G)$ of $F(x)$ and $G(x)$ is the $(m+n)\times(m+n)$ matrix given by
\begin{equation}\label{SylDef}
{\rm Syl}(F,G)=
\begin{pmatrix}
a_0 & a_1 & \cdots & a_m & \\
{} & \ddots & \ddots & {} & \ddots & \\
{} & {} & a_0 & a_1 & \cdots & a_m \\
b_0 & b_1 & \cdots & b_n & \\
{} & \ddots & \ddots & {} & \ddots & \\
{} & {} & b_0 & b_1 & \cdots & b_n 
\end{pmatrix},
\end{equation}
which satisfies
\begin{equation}\label{SylEq}
\det {\rm Syl}(F,G) = a_0^{n}\; b_0^{m} \prod_{1 \le i \le m,\ 1\le j \le n}(\xi_i - \eta_j).
\end{equation}
This determinant is
called the resultant of $F(x)$ and $G(x)$,
which we denote by ${\rm Res}(F,G)$.
For (\ref{SylDef}) and (\ref{SylEq}), see \cite{Syl1840} and \cite[Chapter 12]{GKZ}.
We put $d=\max(m,n)$.
The B\'ezout matrix ${\rm Bez}(F,G)$ of $F(x)$ and $G(x)$ is the $d \times d$ matrix determined by
\begin{equation}\label{BezDef}
\frac{F(x)G(y)-F(y)G(x)}{x-y} = 
\begin{pmatrix}
1 & x & \ldots & x^{d-1}
\end{pmatrix}
{\rm Bez}(F,G)
\begin{pmatrix}
1 \\ y \\ \vdots \\ y^{d-1}
\end{pmatrix},
\end{equation}
which satisfies
\begin{equation}\label{BezEq}
\det {\rm Bez}(F,G) = 
\left.
\begin{cases}
(-1)^{d(d-1)/2} & (m \ge n) \\
(-1)^{d(d-1)/2} (-1)^d (-1)^{mn} & (m \le n)
\end{cases}
\right\}
(a_0 b_0)^d \prod_{1 \le i \le m,\ 1\le j \le n}(\xi_i - \eta_j).
\end{equation}
For (\ref{BezDef}) and (\ref{BezEq}), see \cite{Cay1857} and \cite[Chapter 12]{GKZ}.
The discriminant $\Delta(F)$ of $F(x)$, which was coined by \cite{Syl1851}, is defined by
\begin{equation}
\Delta(F) = a_0^{2m-2} \prod_{1 \le i < j \le m} (\xi_i - \xi_j)^2. \label{DiscDef}
\end{equation}
Then, we have the following relations between resultant and discriminant.
\begin{align}
{\rm Res}(F(x),F'(x))&= (-1)^{m(m-1)/2} \ a_0\ \Delta(F) \label{ResDisc1} \\
{\rm Res}\left(F'(x),F(x)-\frac{1}{m} x F'(x)\right) &= (-1)^{m(m-1)/2}\ \frac{1}{m}\ \Delta(F) \label{ResDisc2} \\
\Delta(FG) &= \Delta(F)\ \Delta(G)\ {\rm Res}(F,G)^2 \label{DiscRes}
\end{align}
%

%
%
%
\subsection{The matrix $M_d(f(x)^e)$ and its determinant in characteristic $p$}
Let $p$ be a prime.
We take integers $r$, $e$, $d$ such that $r \ge 2$, $e \ge 0$, $0 \le d \le p$.
For a generic polynomial $f(x)=s_0 x^r + s_1 x^{r-1} + \cdots + s_r$ of degree $r$ in characteristic $p$, we put $f(x)^e = \sum_{i \ge 0} c_i x^i$ and 
define a matrix $M_d(f(x)^e)$ by
\begin{equation}\label{defMd}
M_d(f(x)^e) = 
\begin{pmatrix}
c_{p-d}  & \cdots & c_{p-2}  & c_{p-1} \\
c_{2p-d} & \cdots & c_{2p-2} & c_{2p-1} \\
\vdots & \reflectbox{$\ddots$} & \vdots & \vdots \\
c_{dp-d} & \cdots & c_{dp-2} & c_{dp-1}
\end{pmatrix}.
\end{equation}
The purpose of this paper is to study the relation between the determinant $\det M_d(f(x)^e)$ of $M_d(f(x)^e)$ and the discriminant $\Delta(f(x))$ of $f(x)$. 
%
%
\subsection{The condition under which $\det M_d(f(x)^e)$ is a power of $\Delta(f(x))$}
Let $p$ be a prime.
We assume that $f(x)$ is a generic monic polynomial of degree $r$ in characteristic $p$.
\begin{equation*}
f(x) = x^r + s_1 x^{r-1} + \cdots + s_r = (x-x_1)\cdots(x-x_r)
\end{equation*}
We put $\delta(x_1,\ldots,x_r) = \prod_{1 \le i < j \le r}(x_i - x_j)$.
We remark that $\delta(x_1,\ldots,x_r)$ is symmetric in $x_1$, \ldots, $x_r$ if and only if $p=2$.

We shall be interested in the following set given for a prime $p$.
\[
{\mathscr A}(p) = \left\{
(r,e,d)\in \Z^3 \  \left| 
\begin{array}{l}
r \ge 2,\ e \ge 1,\ 1 \le d \le p,\\ 
\exists\; \varepsilon_{p,r,e,d}\in\F_p^{\times},\ \exists\; \mbox{positive integer } g,\\ 
\det M_d(f(x)^e) = \varepsilon_{p,r,e,d}\ \delta(x_1,\ldots,x_r)^g
\end{array}
\right.
\right\}
\]
For $(r,e,d)\in{\mathscr A}(p)$, counting the degrees in $x_1$, \ldots, $x_r$ of $\det M_d(f(x)^e)$ and $\delta(x_1,\ldots,x_r)$, we have 
\begin{equation}\label{gdef}
g=\left.\left\{ red -\frac{d(d+1)}{2}(p-1)\right\} \right/ \frac{r(r-1)}{2}.
\end{equation}

We compare the set ${\mathscr A}(p)$ with the following set ${\mathscr B}(p)$.
\begin{align*}
{\mathscr B}_{+}(p) &= \left\{
(r,e,d)\in\Z^3\ |\ 2 \le r \le p,\ e=p-1,\ d=r
\right\} \\
{\mathscr B}_0(p) &= \left\{
(r,e,d)\in\Z^3\ \left|\ 
2 \le r \le p+1,\ 
(p-1)/2 < e \le p-1,\ 
r(p-1-e) \le p-1,\ 
d=r-1
\right.
\right\}\\
{\mathscr B}_{-}(p) &= \left\{
(r,e,d)\in\Z^3\ |\ 2 \le r,\ (p-1)/2 < e \le p-1,\ r(p-1-e)=p-1,\ d=r-2 
\right\} \\
{\mathscr B}(p) &= {\mathscr B}_{+}(p) \cup {\mathscr B}_{0}(p) \cup {\mathscr B}_{-}(p)
\end{align*}
%
%
We remark that if $(r,e,d)\in{\mathscr B}_{+}(p)$ [resp. $(r,e,d)\in{\mathscr B}_{-}(p)$], we have $(r,e,d-1)\in{\mathscr B}_0(p)$ [resp. $(r,e,d+1)\in{\mathscr B}_0(p)$].
In Section 2, we shall prove the following
%
%
\begin{thrm}
For any prime $p$, we have
\[{\mathscr A}(p)\supset{\mathscr B}(p).\]
Moreover, for $(r,e,d)\in{\mathscr B}(p)$, we have:
\begin{align}
&\varepsilon_{p,r,e,d} = 
\begin{cases}
(-1)^{r-1} \,\varepsilon_{p,r,e,d-1} & (when\ (r,e,d)\in{\mathscr B}_{+}(p) \\
(-1)^{\frac{(r+1)(r+2)}{2}+ \frac{r(r+1)}{2} e}\,\{ r(p-1-e)\}!\, e!^r & (when\ (r,e,d)\in{\mathscr B}_0(p))\\
(-1)^r \,\varepsilon_{p,r,e,d+1} & (when\ (r,e,d)\in{\mathscr B}_{-}(p) )
\end{cases}
\\
&g= 2e-(p-1)
\end{align}
\end{thrm}
We encountered the assertion of Theorem 1 for the special case $e=p-2$ in \cite[Lemma 1]{Ku}, which was the motivation of the study in this paper.
Let $C$ be a non-singular projective curve over $\F_{p^2}$ of genus greater than $1$.
By
\cite{Ihara},
there is a one to one correspondence between 
certain deformations of the congruence relation of $C$
and
differential forms $\omega$ of order $p-1$ on $C$ with certain properties, including the property that 
$\omega^{1/(p-1)}$ is an eigenform of the Cartier operator with eigenvalue $1$ or $0$.
Here $\omega^{1/(p-1)}$ is a $1$-form on a finite separable extension of $C$.
Then, the assertion of Theorem 1 for the special case $e=p-2$ was applied to show that, when the genus of $C$ is $2$ and $p \ne 2$, the differential form $\omega$ is
invariant under the action of the hyperelliptic involution of $C$.
In Section 3, we shall prove the following three Theorems.
%
%
\begin{thrm}
For any prime $p$, we have
\[
\left\{ (r,e,d) \in {\mathscr A}(p) \left| \begin{array}{l} r=2 \mbox{\rm\ or\ }\\ r \nmid p-1 \end{array}\right.\right\} =
\left\{ (r,e,d) \in {\mathscr B}(p) \left| \begin{array}{l} r=2 \mbox{\rm\ or\ }\\ r \nmid p-1 \end{array}\right.\right\}.
\]
Here, $r \nmid p-1$ means that $r$ is not a divisor of $p-1$.
\end{thrm}
%
%
%
\begin{thrm}
When $p<200000$, we have ${\mathscr A}(p)={\mathscr B}(p)$.
\end{thrm}
%
%
%
\begin{thrm}
When $p=2\pi+1$ with a prime $\pi$, we have ${\mathscr A}(p)={\mathscr B}(p)$.
\end{thrm}
The basic idea for proving these three Theorems is to observe the matrix $M_d(f(x)^e)$ with $f(x)$ specialized to $x^r-1$, $x^r-x$ and some other polynomials.
We shall introduce a certain type of permutation matrix, which we call the periodic permutation matrix (PPM) (Definition \ref{PPMdef}),
and give the classification theorem of PPM (Proposition \ref{PPMclassification}), which plays an essential role for proving these three Theorems.
Theorem 3 was verified effectively by using SageMath \cite{SM}.
It is an interesting question whether ${\mathscr A}(p)={\mathscr B}(p)$ holds for any prime $p$.

%
%
\subsection{The formula for ${M_d(f(x)^e)}^{-1}M_d(f(x)^{e+1})$}
Suppose we have $(r,e,d)\in{\mathscr B}_0(p)$ such that $(r,e+1,d)\in{\mathscr B}_0(p)$.
We remark that this condition implies $p \ge 5$ and $r\le p-1$.
Then, we shall prove a formula which describes the matrix ${M_d(f(x)^e)}^{-1}M_d(f(x)^{e+1})$ in Section 4. 

To introduce the following several matrices, we now assume that $x_1$, \ldots, $x_r$ are independent variables over $\Q$.
We put $f(x)= x^r + s_1 x^{r-1} + \cdots + s_r =(x-x_1)\cdots(x-x_r)$ and $\varphi(t)= 1 + s_1 t + \cdots + s_r t^r = (1-x_1 t)\cdots(1-x_r t)$.
For $\lambda\in\C$ and $\ell \ge 0$, we define a polynomial $\beta_{\ell}(\lambda)$ of $s_1$, \ldots, $s_r$ by
\(
\varphi(t)^\lambda = \sum_{\ell\ge 0} \; \beta_{\ell}(\lambda) \; t^{\ell}.
\)
We also put $\zeta_i = \dfrac{n}{rn-(r-i)}$ for $1\le i \le r-1$ with $n$ given by $n=p-1-e$.
We put
\begin{align}
B_r &= \label{matBr}
\begin{pmatrix}
{} & {} & 1 \\
{} & \reflectbox{$\ddots$} & {} \\ 
1  & {} & {}
\end{pmatrix}
\;{\rm Bez}\left(f'(x), f(x) - \frac{1}{r} x f'(x)\right)
, \\ 
Q_r &= \label{matQr}
\begin{pmatrix}
1 & & & & \\
\beta_1(-{1}/{r}) & 1 & & & \\
\beta_2(-{1}/{r}) & \beta_1(-{2}/{r}) & \ddots & & \\
\vdots & \vdots & & 1 & \\ 
\beta_{d-1}(-{1}/{r}) & \beta_{d-2}(-{2}/{r}) & \cdots & \beta_1(-{(r-2)}/{r}) & 1
\end{pmatrix},\\
Z_{r,n} &= \label{matZrn}
\begin{pmatrix}
\zeta_1 & & & \\
& \zeta_2 & & \\
& & \ddots & \\
& & & \zeta_{r-1}
\end{pmatrix}, \\
P_r &= \label{matPr}
\begin{pmatrix}
1 & & & & \\
\beta_1\left(-{(r-2)}/{r}\right) & 1 & & & \\
\vdots & & \ddots & & \\ 
\beta_{d-2}\left(-{2}/{r}\right) & \cdots & \beta_1\left(-{2}/{r}\right) & 1 & \\
\beta_{d-1}\left(-{1}/{r}\right) & \cdots & \beta_2\left(-{1}/{r}\right) & \beta_1\left(-{1}/{r}\right) & 1
\end{pmatrix}.
\end{align}
Here we remark that, by (\ref{ResDisc2}),
\begin{equation}
\det B_r =(-1)^{r(r-1)/2}\ \dfrac{1}{r}\ \Delta(f(x))
\end{equation}
holds.
Then, the matrices $B_r$, $Q_r$, $Z_{r,n}$, $P_r$ can be reduced modulo $p$, and the following formula holds.
\begin{thrm}
For $(r,e,d) \in {\mathscr B}_0(p)$ such that $(r,e+1,d)\in{\mathscr B}_0(p)$, we have
\begin{equation}
{M_d(f(x)^e)}^{-1}M_d(f(x)^{e+1}) =B_r \ Q_r \ Z_{r,n} \ P_r.
\end{equation}
\end{thrm}
It seems interesting that the B\'ezout matrix of $f'(x)$ and $f(x)-\frac{1}{r} x f'(x)$ appears in this formula.
An early step of finding and proving Theorem 5 was that, for fixed $r$ and $n$ and for running $p$, we computed the matrix ${M_d(f(x)^e)}^{-1}M_d(f(x)^{e+1})$
as a matrix with components in $\Q[s_1,\ldots,s_r]$.
To do this, 
Wang's algorithm for the rational reconstruction was effective
(cf. \cite{Wang}, \cite{WGD}, \cite{Monagan2004}).
We shall prove Theorem 5 in Section 4.
%
%
%
%
\subsection{Two similar experimental equalities}
In Section 5, we shall give two similar experimental equalities.
The first one is related to a generalization of Theorem 1.
The second one is related to a theorem of Glynn (\cite{Glynn}, Theorem 4.1).
These two equalities remain unproven.
%
%
\subsection{Acknowledgments}
Throughout the study of this paper, a lot of observations involving polynomial, discriminant, matrix and determinant were necessary.
Those observations hardly accepted manual computations.
SageMath \cite{SM} enabled us to make those observations through efficient computation.
We would like to express our gratitude to the SageMath community for their continuous development and support.
%
%
%
%
%
\section{Proof of ${\mathscr A}(p)\supset{\mathscr B}(p)$}
%
%
\subsection{The divisibility of $\det M_d(f(x)^e)$ by power of $\delta(x_1,\ldots,x_r)$}
In general, for a field $K$ of characteristic $p$ and a variable $x$ over $K$, we denote the rational function fields over $K$ of variables $x$ and $x^p$ by $K(x)$ and $K(x^p)$ respectively.
Then, we have
\[
K(x) = \bigoplus_{0 \le i \le p-1} K(x^p) x^{i}.
\]
Suppose we have $F(x) \in K(x)$ and $F_i(x^p) \in K(x^p)$ ($0 \le i \le p-1$) such that
\[
F(x) = \sum_{0 \le i \le p-1} F_i(x^p) x^i,
\]
then we define 
\[
F[t,x] = \sum_{0 \le i \le p-1} F_i(t) x^i,
\]
with another variable $t$.
%
%
\begin{lemma}\label{lemma1}
We put $f(x)=(x-x_1) \cdots (x-x_r) \in \F[x_1,\ldots,x_r][x]$.
For integers $r$, $e$ and $d$, we assume
\begin{equation}\label{theConditionOne}
r \ge 2,\ \ \ \frac{p-1}{2} < e \le p-1,\ \ \ 1 \le d \leq p.
\end{equation}
Then, for $1\le i \le r$ and $1\le j \le r$,
$\left\{\left( \frac{d}{dx} \right)^{p-d} f(x)^e\right\}[x_i^p,x_j]$
is divisible by $\prod_{1\le \ell \le r,\; \ell\ne i}(x_i - x_{\ell})^{2e-(p-1)}$ in $\F_p[x_1, \ldots, x_r]$.
\end{lemma}

\begin{proof}
Let
\[
\frac{1}{\{(x-x_1)\cdots(x-x_r)\}^{p-e}}=\sum_{1\le m \le r} \sum_{1\le k \le p-e} \frac{A_{m,k}}{(x-x_m)^k}
\]
be the partial fraction decomposition.
Then, explicitly we have
\[
A_{m,k}=
\sum_{\substack{h_1,\; \ldots,\; h_{m-1},\; h_{m+1},\;\ldots,\; h_r\ge 0, \\ h_1 + \cdots + h_{m-1} + h_{m+1} + \cdots + h_r = p-e-k}}
\frac
{\prod_{1\le\ell\le r,\; \ell\ne m}\binom{-(p-e)}{h_{\ell}}}
{\prod_{1\le\ell\le r,\; \ell\ne m}(x_m - x_{\ell})^{p-e+h_{\ell}}},
\]
where 
$\binom{-(p-e)}{h_{\ell}}$ denotes the binomial coefficient.
By this, we obtain 
\[
A_{m,k} \left\{ \prod_{1\le\ell\le r,\; \ell\ne m}(x_m-x_{\ell}) \right\}^{2(p-e)-k} \in \F_p[x_1,\ldots,x_r].
\]
Differentiating
\[
f(x)^e
=\{ (x-x_1)\cdots(x-x_r) \}^p \sum_{1\le m \le r} \sum_{1\le k \le p-e} \frac{A_{m,k}}{(x-x_m)^k}
\]
$p-d$ times, we obtain
\begin{equation*}\begin{split}
\left( \frac{d}{dx} \right)^{p-d} f(x)^e 
=
& \sum_{1\le m \le r}\left[ \left\{ \prod_{1\le\ell\le r,\;\ell\ne m}(x-x_{\ell}) \right\}^p\right. 
\\ 
\times & \sum_{1\le k\le p-e} (-1)^{p-d}
\left.\left\{ \prod_{0\le n \le p-d-1}(k+n)\right\} A_{m,k} (x-x_m)^{d-k}\right].
\end{split}\end{equation*}
We remark that $d-k \le d-1 \le p-1$. 
We also remark that when $d-k<0$, we have $k \le p-e < p \le k+p-d-1$.
This implies $\prod_{0\leq n \leq p-d-1}(k+n)=0$ in $\F_p$.
Therefore, we obtain 
\begin{equation*}\begin{split}
\left\{\left( \frac{d}{dx} \right)^{p-d} f(x)^e\right\}[x_i^p,x_j]
=
& \left\{ \prod_{1\le\ell\le r,\; \ell\ne i}(x_i-x_{\ell}) \right\}^p
\\
\times & \sum_{1\le k\le p-e,\ k \le d} (-1)^{p-d}
\left\{ \prod_{0\le n \le p-d-1}(k+n)\right\} A_{i,k} (x_j-x_i)^{d-k}.
\end{split}\end{equation*}
Since
\[
p-\{2(p-e)-k\} = 2e-p+k \ge 2e -p+1,
\]
we obtain the assertion.
\end{proof}
%
%
We put
\begin{equation}
{\mathscr D}(p) = \left\{ (r,e,d) \in \Z^3  \left| 
\begin{array}{l}
r \ge 2, \ \frac{p-1}{2} < e \le p-1, \ 1 \le d \le p,\\
d(p-1) \le re \le (d+1)(p-1)
\end{array}
\right. \right\}.
\end{equation}
We remark that if $d(p-1)\le re$ does not hold, we have $ \det M_d(f(x)^e) =0$.
We also remark that if $re \le (d+1)(p-1)$ holds, we have
\begin{equation}
\left( \frac{d}{dx} \right)^{p-d} f(x)^e= 
\begin{pmatrix}
1 & x^p & \ldots & x^{(d-1)p}
\end{pmatrix}
\widetilde{M_d(f(x)^e)}
\begin{pmatrix}
1 \\
x \\
\vdots \\
x^{d-1}
\end{pmatrix},
\end{equation}
where
\begin{equation*}
\widetilde{M_d(f(x)^e)} = M_d(f(x)^e)
\begin{pmatrix}
{}_{p-d}{\rm P}_{p-d} & {}     & {}                    & {} \\
{}                    & \ddots & {}                    & {} \\
{}                    & {}     & {}_{p-2}{\rm P}_{p-d} & {} \\
{}                    & {}     & {}                    & {}_{p-1}{\rm P}_{p-d}
\end{pmatrix}
\end{equation*}
with ${}_m{\rm P}_{\ell} = m (m-1) \cdots (m-\ell+1)$.
For $(r,e,d)\in{\mathscr D}(p)$, it is immediate to see the followings.
\begin{align}
& r \le 2(p-1) \ \ (p\ne 2) \label{Dprop1} \\
& re \le p^2-1 \label{Dprop2} \\
& (r-d-1)(p-1) \le r(p-1-e) \le (r-d)(p-1) \label{Dprop3} \\
& 0 \le r-d \le \max(p-2,1) \label{Dprop4}
\end{align}
As for (\ref{Dprop4}), when $p \ge 3$, we have
\begin{align*}
e\{p-2-(r-d)\} 
&= e(p-2+d)-re \\
&\ge \frac{p+1}{2} (p-2+d) - (d+1)(p-1)\\
&= \frac{1}{2} (p-3)(p-d) \\
&\ge 0.
\end{align*}
For $0 \le m \le \max(p-2,1)$, we put
\begin{align}
D_m(p) &= \left\{ (r,e)\in\Z^2 \left| 
\begin{array}{l}
r \ge 2,\ \frac{p-1}{2} < e \le p-1,\ r \le p+m, \\
(m-1)(p-1) \le r(p-1-e) \le m(p-1)
\end{array}
\right.\right\},
\\
{\mathscr D}_m(p) &= \left\{ (r,e,d)\in \Z^3 \left|\ (r,e)\in D_m(p),\ d=r-m \right.\right\}.
\end{align}
Then, we have 
\begin{equation}
{\mathscr D}(p) = \bigcup_{0 \le m \le \max(p-2,1)} {\mathscr D}_m(p).
\end{equation}
We have the followings.
\begin{align*}
{\mathscr D}_0(p) &= \{ (r,p-1,r) |\ 2 \le r \le p \} \\
{\mathscr D}_{p-2}(p) &= \{ (2(p-1),\frac{p+1}{2},p) \} \ \ (p \ne 2,3) \\
D_0(p) &\subset D_1(p)
\end{align*}
The sets $D_m(p)$ are illustrated as follows.
\begin{center}
\includegraphics{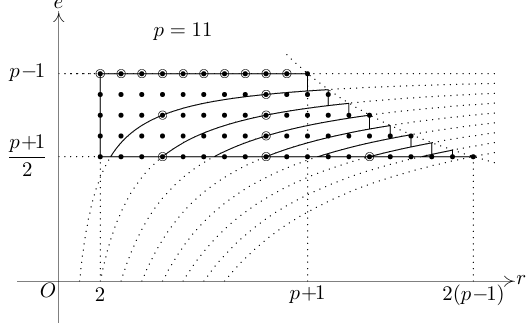}
\captionof{figure}{$D_m(p)$ for $p=11$}
\end{center}
Here, points $(r,e)\in D_m(p)\cap D_{m+1}(p)$ are shown as circled dots. 
For $(r,e)\in D_m(p)\cap D_{m+1}(p)$, we have $re=(r-m)(p-1)$, and
\begin{equation}\label{reMMPlus1}
M_{r-m}(f(x)^e) = \begin{pmatrix}
\ast   & \multicolumn{3}{c}{\multirow{3}{*}{$M_{r-m-1}(f(x)^e)$}} \\
\vdots & {} & {}     & {} \\
\ast   & {} & {}     & {} \\
1      & 0  & \cdots & 0
\end{pmatrix}.
\end{equation}
Hence, we have $\det M_{r-m}(f(x)^e) = (-1)^{r-m-1} \det M_{r-m-1}(f(x)^e)$.
\begin{lemma}\label{lemma2}
For $(r,e,d)\in{\mathscr D}(p)$, $\det M_d(f(x)^e)$ is divisible by $\delta(x_1,\ldots,x_r)^{2e-(p-1)}$.
\end{lemma}
\begin{proof}
We put
\begin{equation}\label{defN}
N=
\begin{pmatrix}
1 & x_1^p & \cdots & x_1^{(d-1)p} \\
1 & x_2^p & \cdots & x_2^{(d-1)p} \\
\vdots & \vdots & \ddots & \vdots \\
1 & x_d^p & \cdots & x_d^{(d-1)p} \\
\end{pmatrix}
\widetilde{M_d(f(x)^e)}
\begin{pmatrix}
1         & 1         & \cdots & 1         \\
x_1       & x_2       & \cdots & x_d       \\
\vdots    & \vdots    & \ddots & \vdots    \\
x_1^{d-1} & x_2^{d-1} & \cdots & x_d^{d-1} 
\end{pmatrix}.
\end{equation}
Then, the $(i,j)$-component of $N$ is equal to 
$\left\{\left( \frac{d}{dx} \right)^{p-d} f(x)^e\right\}[x_i^p,x_j]$
for $1\le i \le d$, $1 \le j \le d$.
First, we assume $r > d$.
By Lemma \ref{lemma1}, the $i$-th row of $N$ is divisible by $\{\prod_{1\le\ell\le r,\;\ell\ne i}(x_i - x_{\ell})\}^{2e-(p-1)}$.
Hence, $\det N$ is divisible by $\{\prod_{1\le\ell\le r,\;\ell\ne i}(x_i - x_{\ell})\}^{2e-(p-1)}$ for any $1 \le i \le d$. 
In (\ref{defN}), the rightmost and the leftmost matrices in the right hand side are Vandermonde matrices, which do not contain $x_r$.
Hence, $\det \widetilde{M_d(f(x)^e)}$ is divisible by $(x_i - x_r)^{2e-(p-1)}$ for any $1 \le i \le d$.
Since $\det \widetilde{M_d(f(x)^e)}$ is symmetric in $x_1$, \ldots, $x_r$, $\det \widetilde{M_d(f(x)^e)}$ is divisible by $\delta(x_1,\ldots,x_r)^{2e-(p-1)}$.
Hence, $\det M_d(f(x)^e))$ is divisible by $\delta(x_1,\ldots,x_r)^{2e-(p-1)}$.
Second, in the case $r=d$, since $D_0(p) \subset D_1(p)$, by (\ref{reMMPlus1}) we obtain the assertion.
\end{proof}
%
%
%
%
We gave the set ${\mathscr B}(p)={\mathscr B}_{+}(p)\cup{\mathscr B}_0(p)\cup{\mathscr B}_{-}(p)$ in Introduction.
Now we have
\begin{align*}
{\mathscr B}_{+}(p) &= {\mathscr D}_0(p), \\
{\mathscr B}_0(p) &= {\mathscr D}_1(p), \\
{\mathscr B}_{-}(p) &= \{ (r,e,d)\in {\mathscr D}_2(p) |\ r(p-1-e)=p-1 \}.
\end{align*}
In the following Lemma, we consider the degree such that $\deg(x_i)=1$.
\begin{lemma}\label{lemma3}
For $(r,e,d)\in{\mathscr D}(p)$, we have
\[
\deg \det M_d(f(x)^e) = \deg \delta(x_1,\ldots,x_r)^{2e - (p-1)} \Longleftrightarrow (r,e,d)\in{\mathscr B}(p). 
\]
Hence, for $(r,e,d) \in {\mathscr B}(p)$, we have
\[
\det M_d(f(x)^e) = \varepsilon_{p,r,e,d}\ \delta(x_1,\ldots,x_r)^{2e-(p-1)}, \ \ \varepsilon_{p,r,e,d}\in\F_p.
\]
\end{lemma}
\begin{proof}
By
\begin{align*}
& \deg \det M_d(f(x)^e) -  \deg \delta(x_1,\ldots,x_r)^{2e-(p-1)}  \\
=
& \left\{ red - (p-1)\frac{d(d+1)}{2} \right\} - \frac{r(r-1)}{2} \{ 2e-(p-1) \} \\
=
& (p-1) (r-d-1) \left\{ \frac{r-d}{2} - \left(\frac{re}{p-1} - d\right) \right\},
\end{align*}
the assertion follows.
\end{proof}
%
The constant $\varepsilon_{p,r,e,d}\in\F_p$ shall be computed in what follows to show $\varepsilon_{p,r,e,d}\in\F_p^{\times}$.
%
%
%
%
\subsection{Periodic Permutation Matrix (PPM)}
Let $M_d((x^r-1)^e)$ and $M_d((x^r-x)^e)$ be the matrices $M_d(f(x)^e)$ with $f(x)$ specialized to $f(x)=x^r-1$ and $f(x)=x^r-x$ respectively.
In the following Subsections, we shall see that, in some cases, the matrices $M_{d}((x^r-1)^e)$ and $M_{d}((x^r-x)^e)$ give PPM's in the sense defined in the following

\begin{defi}[Periodic Permutation Matrix (PPM)]\label{PPMdef}
Let $h$, $k$ be integers such that $h\ge 1$, $k\ge 1$ and $\gcd(h,k)=1$. 
Let $M$ be a permutation matrix of size $d \ge 1$, i.e., there exists a permutation
$\sigma=\sigma_M$ of the set $\{1, \ldots, d\}$ such that 
\[
M_{i,j} = 
\begin{cases}
1 \ \ (\ j=\sigma(i)\ ) \\
0 \ \ (otherwise)
\end{cases}
\ \ (1 \le i,j \le d).
\]
We call M a periodic permutation matrix (PPM) of type $(h,k)$ of size $d$ if
\[
\sigma(i+1) - \sigma(i) = -h \ or \ k 
\]
holds for any $i$ such that $1 \le i \le d-1$.
When $\sigma(i+1)-\sigma(i)=-h$ [resp. $k$], we say that a left [resp. right] move occurs at $i$.
\end{defi}

%
%
When $d=h+k$, for each $m \in \{1,2,\ldots,d\}$, there exists a unique PPM of type $(h,k)$ of size $d$ such that $\sigma(1)=m$.
We denote this matrix by $A_m=A_m(h,k)$.
For $A_m(h,k)$, the permutation $\sigma$ is given by
\[
\sigma(i) = ki+m-k
\]
for $1 \le i \le d$, if we use the identification $\{1,2,\ldots,d\} \ni i \mapsto i \bmod d \in \Z/d\Z$. 
Any PPM of type $(h,k)$ of size $d=h+k$ is equal to $A_m(h,k)$ for some $m \in \{1,2,\ldots,d\}$.
In the following Proposition, we denote
\[
A_1 = B_d,\ A_2 = B_{d-1},\ \ldots,\ A_k=B_{h+1},\ A_{k+1}=B_h,\ \ldots,\ A_{d-1}=B_2,\ A_d=B_1.
\]
We also have a matrix $K=K(h,k)$ with the following properties
\[
A_1=
\begin{pmatrix}
1 \\
{} & K
\end{pmatrix}
,\ 
A_k=
\begin{pmatrix}
K \\
{} & 1
\end{pmatrix}
,\ 
B_h=
\begin{pmatrix}
{} & K \\
1
\end{pmatrix}
,\ 
B_1=
\begin{pmatrix}
{} & 1 \\
K
\end{pmatrix}
.
\]
We put
\[
{\rm I}_d =
\begin{pmatrix}
1 \\
{} & \ddots \\
{} & {} & 1
\end{pmatrix}
,\
{\rm J}_d =
\begin{pmatrix}
{} & {} & 1 \\
{} & \reflectbox{$\ddots$} \\
1 
\end{pmatrix}
\ \ (d\times d \text{ matrices})
.
\]
Then, ${\rm I}_d$ and ${\rm J}_d$ are PPM's of type $(h,1)$ and $(1,k)$ for any $h \ge 1$ and $k \ge 1$ respectively.
%
%
\begin{prop}\label{PPMclassification}
Let $h$, $k$, $d$ be integers such that $h\ge1$, $k\ge1$, $\gcd(h,k)=1$, $d\ge1$.
Then, the PPM's of type $(h,k)$ of size $d$ are classified as follows.
In particular, there exits no PPM such that $h\ge 2$, $k\ge 2$ and $d \ne 0,\ \pm 1 \bmod h+k$.
\begin{itemize}
%
%
\item[]
Case $h \ge 2$, $k \ge 2$: 
%
\begin{itemize}
\item[] $d = \ell(h+k)+1 \ \ (\ell \ge 0)$
\item[] $
\begin{pmatrix}
A_1 \\
{} & \ddots \\
{} & {} & A_1 \\
{} & {} & {} & {\rm I}_1
\end{pmatrix}
,\ 
\begin{pmatrix}
{} & {} & {} & B_1 \\
{} & {} & \reflectbox{$\ddots$} \\
{} & B_1 \\
{\rm I}_1
\end{pmatrix}
$
\end{itemize}
%
%
\item[]
\begin{itemize}
\item[] $d = \ell(h+k)+h+k-1 \ \ (\ell \ge 0)$ 
\item[] $
\begin{pmatrix}
A_k \\
{} & \ddots \\
{} & {} & A_k \\
{} & {} & {} & K
\end{pmatrix}
,\ 
\begin{pmatrix}
{} & {} & {} & B_h \\
{} & {} & \reflectbox{$\ddots$} \\
{} & B_h \\
K
\end{pmatrix}
$
\end{itemize}
%
%
\item[]
\begin{itemize}
\item[] $d = \ell(h+k) \ \ (\ell \ge 1)$ 
\item[] $
\begin{pmatrix}
A_1 \\
{} & \ddots \\
{} & {} & A_1 
\end{pmatrix}
,\ldots, 
\begin{pmatrix}
A_k \\
{} & \ddots \\
{} & {} & A_k 
\end{pmatrix}
$
,
$
\begin{pmatrix}
{} & {} & B_h \\
{} & \reflectbox{$\ddots$} \\
B_h 
\end{pmatrix}
,\ldots, 
\begin{pmatrix}
{} & {} & B_1 \\
{} & \reflectbox{$\ddots$} \\
B_1
\end{pmatrix}
$
\end{itemize}
\end{itemize}
%
%
%
\begin{itemize}
\item[] Case $h=1$, $k \ge 1$: (Hence,\ $B_1(1,k)={\rm J}_{1+k}$ and $K(1,k)={\rm J}_k$.)
%
\item[]
\begin{itemize}
\item[] $d = \ell(1+k)+m \ \ (\ell \ge 0,\ 1 \le m \le k)$ 
%
\item[] $
\begin{pmatrix}
A_m \\
{} & \ddots \\
{} & {} & A_m \\
{} & {} & {} & {\rm J}_m
\end{pmatrix}
,\ 
\begin{pmatrix}
{} & {} & {} & B_1 \\
{} & {} & \reflectbox{$\ddots$} \\
{} & B_1 \\
{\rm J}_m
\end{pmatrix}
$
\end{itemize}
%
\item[]
\begin{itemize}
\item[] $d = \ell(1+k) \ \ (\ell \ge 1)$ 
\item[] $
\begin{pmatrix}
A_1 \\
{} & \ddots \\
{} & {} & A_1 
\end{pmatrix}
,\ldots, 
\begin{pmatrix}
A_k \\
{} & \ddots \\
{} & {} & A_k 
\end{pmatrix}
$
,
$
\begin{pmatrix}
{} & {} & B_1 \\
{} & \reflectbox{$\ddots$} \\
B_1
\end{pmatrix}
$
\end{itemize}
\end{itemize}
%
%
\begin{itemize}
\item[] Case $h\ge1$, $k=1$:  (Hence,\ $A_1(h,1)={\rm I}_{h+1}$ and $K(h,1)={\rm I}_h$.)
%
\item[]
\begin{itemize}
\item[] $d = \ell(h+1)+m \ \ (\ell \ge 0, 1 \le m \le h)$ 
%
\item[] $
\begin{pmatrix}
A_1 \\
{} & \ddots \\
{} & {} & A_1 \\
{} & {} & {} & {\rm I}_m
\end{pmatrix}
,\ 
\begin{pmatrix}
{} & {} & {} & B_m \\
{} & {} & \reflectbox{$\ddots$} \\
{} & B_m \\
{\rm I}_m
\end{pmatrix}
$
\end{itemize}
%
%
\item[]
\begin{itemize}
\item[] $d = \ell(h+1) \ \ (\ell \ge 1)$ 
\item[] $
\begin{pmatrix}
A_1 \\
{} & \ddots \\
{} & {} & A_1 
\end{pmatrix}
,
\begin{pmatrix}
{} & {} & B_h \\
{} & \reflectbox{$\ddots$} \\
B_h
\end{pmatrix}
$
,
\ldots
,
$
\begin{pmatrix}
{} & {} & B_1 \\
{} & \reflectbox{$\ddots$} \\
B_1
\end{pmatrix}
$
\end{itemize}
\end{itemize}
%
Here, in all of the three cases, when $\ell=0$, the two matrices coincide.
\end{prop}
%
%
\begin{proof}
Suppose $h\ge 1$, $k\ge 1$, $d\ge 1$ such that $\gcd(h,k)=1$ are given.
Let $M$ be a PPM of type $(h,k)$ of size $d$.
First, we assume $d<h+k$.

\noindent
Case $d=1$:
We should have $M={\rm I}_1$.

\noindent
Case $d\ge 2$, $h=1$:
Since $d \le k$, no right move occurs.
Hence, we have $M={\rm J}_d$.

\noindent
Case $d\ge 2$, $k=1$:
Similarly, we have $M={\rm I}_d$.

\noindent
Case $d\ge 2$, $h\ge 2$, $k\ge 2$:
Both right moves and left moves should occur.
Hence, $h<d$ and $k<d$.
By this, we obtain
\[
1 \le d-k < d-k+1 \le h < h+1 \le d.
\] 
Suppose we have $i\in\{1,2,\ldots,d\}$ such that $d-k+1 \le \sigma(i) \le h$.
Then, by $\sigma(i) -h \le 0$, a left move does not occur at $i$.
Moreover, by $d+1\le \sigma(i)+k$, a right move does not occur at $i$. 
Hence, we should have $i=d$.
This means that $d-k+1=\sigma(d)=h$.
Hence, we obtain $d=h+k-1$ and
\[
\begin{pmatrix}
M & {} \\
{} & 1
\end{pmatrix}
= A_k.
\]
Hence, we obtain $M=K$.
Second, we assume $d>h+k$.
We take four $(h+k)\times (h+k)$ matrices at four corners of $M$ as follows.
\[
\begin{pmatrix}
M_{11} & M_{12} \\
M_{21} & M_{22}
\end{pmatrix}.
\]
These matrices may intersect with each other.

We shall show that $M_{11}$ or $M_{21}$ is a PPM of type $(h,k)$.
We put
\begin{align*}
&i_1 = \min\{ 1 \le i \le d \mid 1 \le \sigma(i) \le h+k \}, \\
&i_2 = \max\{ i_1 \le i \le d \mid \{ \sigma(i_1), \sigma(i_1+1), \ldots, \sigma(i) \} \subset \{1, 2, \ldots, h+k\} \}.
\end{align*}
If $i_2=d$, then $M_{21}$ is a PPM of type $(h,k)$.
So we assume $i_2<d$.
By
\begin{align}
\sigma(i_2) \le h+k, \label{PPMProof1} \\
\sigma(i_2+1) > h+k, \label{PPMProof2}
\end{align}
we obtain
\begin{equation}
\sigma(i_2+1) = \sigma(i_2) + k. \label{PPMProof3}
\end{equation}
By (\ref{PPMProof1}), (\ref{PPMProof2}), (\ref{PPMProof3}), we obtain $1 \le \sigma(i_2)-h \le k$.
We take $i_3 \in \{1,2,\ldots,d\}$ such that
\begin{equation}
\sigma(i_3) = \sigma(i_2) - h. \label{PPMProof4}
\end{equation}
We claim that $i_3=1$.
If $i_3 > 1$, we have
\begin{equation}
\sigma(i_3) = \sigma(i_3 - 1) + \left\{ \begin{matrix} -h \\ k \end{matrix} \right\} . \label{PPMProof5}
\end{equation}
Here, $\left\{ \begin{matrix} -h \\ k \end{matrix}\right\}$ stands for $-h$ or $k$.
Hence, by (\ref{PPMProof4}), we have
\begin{equation}
\sigma(i_2) = \sigma(i_3-1) + \left\{ \begin{matrix} 0 \\ h+k \end{matrix} \right\}. \label{PPMProof6}
\end{equation}
By (\ref{PPMProof1}), $h+k$ does not occur in (\ref{PPMProof6}).
Then, we have $\sigma(i_2) = \sigma(i_3-1)$. Hence, $i_2 = i_3-1$.
By (\ref{PPMProof3}), (\ref{PPMProof4}), this implies
\[
\sigma(i_2) + k = \sigma(i_2 + 1) = \sigma(i_3) = \sigma(i_2) -h.
\]
This is a contradiction.
Hence, we should have $i_1=i_3=1$.
Then, we have
\[
\sigma(1) + \underbrace{
\left\{ \begin{matrix} -h \\ k \end{matrix} \right\}
+\cdots+
\left\{ \begin{matrix} -h \\ k \end{matrix} \right\}
}_{\text{$i_2 - 1$ times}} = \sigma(i_2) 
\]
By (\ref{PPMProof4}), we obtain
\[
i_2 k = 0 \bmod h+k.
\]
This implies $i_2 = h+k$.
Hence, $M_{11}$ is a PPM of type $(h,k)$.

Thus, we obtained that $M_{11}$ or $M_{21}$ is a PPM of type $(h,k)$.
Similarly, we obtain that $M_{12}$ or $M_{22}$ is a PPM of type $(h,k)$.
By these,
we inductively obtain the assertion for the case $d > h+k$.
\end{proof}
We shall apply Proposition \ref{PPMclassification} to prove Theorem 2 in Lemma \ref{UinB2}. 

%
%
\begin{defi}
For integers $d \ge 1$ and $k$ such that $\gcd(d,k)=1$, we denote
the signature of the bijection $\Z/d\Z \ni x \mapsto kx \in \Z/d\Z$ by $\left[ \dfrac{k}{\;d\;} \right]$.
\[
\left[ \frac{k}{\;d\;} \right] = {\rm sgn} (\Z/d\Z \xrightarrow{\times k} \Z/d\Z)
\]
\end{defi}
%
%
\begin{prop}
For integers $d \ge 1$ and $k$ such that $\gcd(d,k)=1$, we have 
\[
\left[ \frac{k}{\;d\;} \right] =
\begin{cases}
\left( \frac{k}{\;d\;} \right) & (\text{when $d$ is odd}) \\
(-1)^{\frac{k-1}{2} \frac{d-2}{2}}& (\text{when $d$ is even})
\end{cases}
.
\]
Here, $\left( \frac{k}{\;d\;} \right)$ denotes the Jacobi symbol.
In particular, we have
\begin{align}
\left[ \frac{-1}{\;d\;} \right] &= (-1)^{\frac{(d-1)(d-2)}{2}} \label{recipro1} \\
\left[ \frac{2}{\;d\;} \right] &= (-1)^\frac{d^2-1}{8} \label{recipro2} \\
\left[ \frac{p}{\;d\;} \right] &= (-1)^{\frac{p-1}{2}\frac{(d-1)(d-2)}{2}}\left( \frac{d}{\;p\;} \right)^d \ \ (\text{$p$ is an odd prime}) \label{recipro3}
\end{align}
\end{prop}
\begin{proof}
In general, for finite sets $X$ and $Y$ and bijections $\sigma:X\rightarrow X$ and $\tau:Y\rightarrow Y$, we have
\[
{\rm sgn} (\sigma \times \tau:X\times Y \rightarrow X\times Y)=
{\rm sgn} (\sigma : X \rightarrow X)^{|Y|}\cdot 
{\rm sgn} (\tau   : Y \rightarrow Y)^{|X|}.
\]
By this formula with the Chinese remainder theorem, the assertion follows from the following two facts.
\begin{align}
{\rm sgn} (\Z/d\Z \xrightarrow{\times k} \Z/d\Z) &= \left( \frac{k}{\;d\;} \right) \ (\text{when $d\ge 1$ is odd and $\gcd(k,d)=1$}) \label{sgn1}
\\
{\rm sgn} (\Z/2^f\Z \xrightarrow{\times k} \Z/2^f\Z) &= \label{sgn2}
\begin{cases}
1 & (\text{when $f=1$ and $k$ is odd}) 
\\
(-1)^{(k-1)/2} & (\text{when $f \ge 2$ and $k$ is odd})
\end{cases}
\end{align}
The equality (\ref{sgn1}) is known (cf. \cite{Zolotareff}, \cite{Rousseau}, \cite{Sz}).
The equality (\ref{sgn2}) can be verified by using
\[
\left( \Z / 2^f \Z \right)^{\times} = \langle 5 \rangle \times \langle -1 \rangle \ \ \ \ (f\ge 3),
\]
where $\langle 5 \rangle$ and $\langle -1 \rangle$ denote the cyclic groups generated by $5$ and $-1$ respectively.
\end{proof}
%
%
\begin{prop}\label{PPMdet}
Let $h$, $k$, $m$ be integers such that $h \ge 1$, $k \ge 1$, $\gcd(h,k)=1$, $1 \le m \le d=h+k$.
Then, we have
\[
\det A_m(h,k) =(-1)^{(d-1)(m-1)} \left[ \frac{k}{\;d\;} \right]
.
\]
In particular, we have
\[
\det K(h,k) =\left[ \frac{k}{\;d\;} \right]
.
\]
\end{prop}
\begin{proof}
Using the identification $\{1,2,\ldots,d\} \ni i \mapsto i \bmod d \in \Z/d\Z$,
we obtain
\begin{equation*}
\det A_m(h,k) 
=
{\rm sgn} (\Z/d\Z \xrightarrow{+(m-k)} \Z/d\Z) 
\cdot
{\rm sgn} (\Z/d\Z \xrightarrow{\times k} \Z/d\Z).
\end{equation*}
Then,
\begin{align*}
&
{\rm sgn} (\Z/d\Z \xrightarrow{+1} \Z/d\Z)
=
(-1)^{d-1}, \\
& 
\gcd(h,k)=1 \Rightarrow (h+1)(k+1)=0 \bmod 2
\end{align*}
imply the assertion.
\end{proof}
%
%
%
%
\subsection{The constant $\varepsilon_{p,r,e,d}$\ }
In this Subsection, we shall compute $\varepsilon_{p,r,e,d}$ for a prime $p$ and $(r,e,d)\in {\mathscr B}_0(p)$
to complete the proof of Theorem 1.
To do this, we specialize the generic monic polynomial $f(x)=(x-x_1)\cdots(x-x_r)$ of degree $r$ to $f(x)=x^r-1$ [resp. $f(x)=x^r - x -1$ ] 
when $p \nmid r$ [resp. $p=r$].
We also denote the discriminant of $f(x)$ by $\Delta(x^r-1)$ [resp. $\Delta(x^r-x-1)$] under this specialization.

We remark that
\begin{align}
\Delta(x^r-1) &= (-1)^{(r-1)(r-2)/2}\ r^r, \label{Deltaxr1} \\
\Delta(x^r-x-1) &= (-1)^{r(r+1)/2}, \label{Deltaxrx1}
\end{align}
where (\ref{Deltaxr1}) is valid whenever $r \ge 2$, and
(\ref{Deltaxrx1}) is valid whenever $p \mid r \ge 2$.

%
%
\begin{prop}\label{epsilon_formula1}
For $(r,e,d)\in{\mathscr B}_0(p)$ such that $p \nmid r$, we have
\begin{equation}
\varepsilon_{p,r,e,d}=(-1)^{\frac{(r+1)(r+2)}{2}+ \frac{r(r+1)}{2} e}\,\{ r(p-1-e)\}!\, e!^r . \label{epsilonForB0p} 
\end{equation}
\end{prop}
\begin{proof}
We put $M=M_{d}((x^r-1)^e)$.
We shall compute $\det M$.
We have
\[
(x^r-1)^e = \sum_{0 \le m \le e} (-1)^{e-m} \binom{e}{m} x^{mr}.
\]
By the definition of ${\mathscr B}_0(p)$, we have $d=r-1$.
For $1 \le i \le d$, we have a unique $m_i \in \Z$ such that $ip - d \le m_i r \le ip$.
By $p \nmid r$, we have $m_i r \ne ip$.
We have
\begin{align*}
& m_i r \ge ip-d \ge p-d=p+1-r \ge 0, \\
& m_i r \le ip-1 \le dp-1=(r-1)(p-1)+r-2 \le re + r -2 < re + r.
\end{align*}
Hence we have $0 \le m_i \le e$.
By $e \le p-1$, we have $\binom{e}{m_i} \ne 0$ in $\F_p$.
Hence, each row of $M$ contains only one non-zero component.
Since
\[
m_{i_1} r = m_{i_2} r \bmod p \ \ \Longrightarrow\ \  m_{i_1} = m_{i_2} \ \ \Longrightarrow\ \  i_1 = i_2,
\]
each column of $M$ has only one non-zero component.
Therefore, $M$ gives a permutation matrix, which we denote by $\widetilde{M}$.
The permutation $\sigma$ of $\widetilde{M}$ is given by
\[
\sigma(i) = m_i r - (ip - d) + 1\ \ \ (1 \le i \le d).
\]
Hence, for $1 \le i \le d-1$, we have
\[
\sigma(i+1) - \sigma(i) = (m_{i+1} - m_i) r - p.
\]
We take integers $h$ and $k$ defined by
\[
h = p \bmod r,\ \ 0 \le h \le r-1,\ \ k = -p \bmod r,\ \ 0 \le k \le r-1.
\]
Then, we obtain
\[
h \ge 1,\ \ k \ge 1,\ \ \gcd(h,k)=1,\ \ h+k=r.
\]
We also have
\[
|\sigma(i+1)-\sigma(i)| \le d-1 = r-2 \le r-1.
\]
Hence, $\widetilde{M}$ is a PPM of type $(h,k)$ of size $h+k-1$.
Hence, by Proposition \ref{PPMclassification} and Proposition \ref{PPMdet}, we obtain 
\[
\widetilde{M} = K(h,k),\ \ \det\widetilde{M} =
\left[ \frac{-p}{\;r\;} \right].
\]
By this, we have
\[
\det M = \left[ \frac{-p}{\;r\;}\right] \prod_{1 \le i \le d} (-1)^{e-m_i} \binom{e}{m_i}.
\]
In particular, we have $\varepsilon_{p,r,e,d} \ne 0$.
Hence, we have $\varepsilon_{p,r,e,d}=1$ when $p=2$.
In what follows, we assume $p \ne 2$.

Since $M$ gives a permutation matrix, we have
\[
\sum_{1 \le i \le d} m_i r = \sum_{1 \le i \le d} (ip-d) + \sum_{1 \le j \le d} (j-1) = (p-1) \frac{d(d+1)}{2}.
\]
Hence,
\begin{equation}
\sum_{1 \le i \le d} m_i = \frac{p-1}{2} d.\label{result1}
\end{equation}
Now, we compute $\prod_{1 \le i \le d}\binom{e}{m_i}$ inductively on $e =p-1$, $p-2$, \ldots, $e$, i.e., $n=0$, $1$, \ldots, $p-1-e$ if we put $n=p-1-e$.
We remark that $m_i$'s are common for these $e$'s.
When $e=p-1$, we have
\begin{align*}
\prod_{1 \le i \le d}\binom{p-1}{m_i} 
&=
\prod_{1 \le i \le d} \frac{p-1}{1} \frac{p-2}{2} \cdots \frac{p-m_i}{m_i} \\
&=
\prod_{1 \le i \le d} (-1)^{m_i} \\
&=
(-1)^{\frac{p-1}{2} d}.
\end{align*}
Furthermore, we have
\begin{align*}
\frac{\prod_{1 \le i \le d} \binom{e}{m_i} } {\prod_{1 \le i \le d} \binom{e+1}{m_i}} 
&=
\prod_{1 \le i \le d} \frac{e+1-m_i}{e+1} \\
&=
\prod_{1 \le i \le d} \frac{nr+m_i r}{nr} \\
&=
\frac{\{nr-(r-1)\} \{nr-(r-2)\} \cdots \{nr-1\}}{(nr)^d} \\
&=
\frac{\{nr-(r-1)\} \{nr-(r-2)\} \cdots \{nr-1\}nr}{(nr)^r}.
\end{align*}
By these, we obtain
\begin{equation}
\prod_{1 \le i \le d} \binom{e}{m_i} = (-1)^{\frac{p-1}{2} d} \frac{(nr)!}{(n!)^r r^{nr}}. \label{result2}
\end{equation}
Here, we remark that, by $e \le p-1$ and $re \ge (p-1)(r-1)$, we have $0 \le nr \le p-1$.
We remark that $(p-1-e)!\,e! = (-1)^{e+1}$ in $\F_p$.
By (\ref{recipro1}), (\ref{recipro3}), (\ref{Deltaxr1}), (\ref{result1}), (\ref{result2}) and $\frac{g}{2} = e -\frac{p-1}{2}$, we obtain the assertion (\ref{epsilonForB0p}).
\end{proof}

%
%
\begin{prop}\label{epsilon_formula2}
For $(r,e,d)\in{\mathscr B}_0(p)$ such that $p \mid r$, we have
$r=p$, $e=p-1$, $d=p-1$, and
\[
\varepsilon_{p,r,e,d} = \begin{cases}
1 & p =1,\ 2 \bmod 4 \\
-1 & p = 3 \bmod 4
\end{cases}.
\]
\end{prop}
\begin{proof}
By the condition of ${\mathscr B}_0(p)$, $p \mid r$ implies that $r=p$, $e=p-1$, $d=p-1$.
We put $M=M_{d}((x^p - x -1)^{p-1})$.
We have
\[
(x^p - x - 1)^{p-1} = \sum_{0 \le m \le p-1} \binom{p-1}{m} x^{mp} (-x-1)^{p-1-m}.
\]
The matrix $M$ is of the form
\[
M = \begin{pmatrix}
\ast & \cdots & \ast \\
\vdots & \reflectbox{$\ddots$} \\
\ast
\end{pmatrix}.
\]
By this, we obtain
\[
\det M = \binom{p-1}{0} \binom{p-1}{1} \cdots \binom{p-1}{p-2} = \begin{cases}
1 & p=1,\ 2 \bmod 4 \\
-1 & p=3 \bmod 4
\end{cases}.
\]
By (\ref{Deltaxrx1}) and $g=p-1$, we have $\Delta(x^p-x-1)^{g/2}=1$.
Hence, we obtain the assertion, which fits into (\ref{epsilonForB0p}).
\end{proof}
%

%
This completes the proof of Theorem 1.
%
%
%
%
%
\section{Proof of ${\mathscr A}(p)={\mathscr B}(p)$ in some cases}
%
%
\subsection{The superset ${\mathscr U}(p)$ of ${\mathscr A}(p)$, Proof of Theorem 2}
%
%
For a prime $p$, we put
\[
{\mathscr U}(p) = \left\{
(r,e,d)\in\Z^3 \left|
\begin{array}{l}
r \ge 2,\ e \ge 1,\ 1 \le d \le p,\\
d(p-1) \le re \le r(p-1), \\
g >0, \\  
g \in2\Z \text{ when $p\ne 2$ },\ g \in \Z \text{ when $p=2$}
\end{array}
\right.
\right\},
\]
where $g$ is defined by (\ref{gdef}).
We remark that $e \le p-1$ and $d \le r$ hold for $(r,e,d)\in{\mathscr U}(p)$.
%

%
%
\begin{lemma}
For a prime $p$, we have ${\mathscr U}(p) \supset {\mathscr A}(p)$.
\end{lemma}
\begin{proof}
For $(r,e,d)\in{\mathscr A}(p)$,
the power $\delta(x_1,\ldots,x_r)^g$ is symmetric in $x_1$, \ldots, $x_r$.
When $p \ne 2$, this implies that $g$ is even.

The lowest row of $M_{d}(f(x)^e)$ contains the coefficients of $x^{dp-d}$, \ldots, $x^{dp-1}$.
Since $\det M_{d}(f(x)^e) \ne 0$, we should have $d(p-1) \le re$.

If $e \ge p$, we have
\[
M_{d}(f(x)^e) =
\begin{pmatrix}
s_r^p \\
s_{r-1}^p & s_r^p \\
\vdots & \ddots & \ddots \\
s_{r-d+1}^p & \cdots & s_{r-1}^p & s_r^p
\end{pmatrix}
M_{d}(f(x)^{e-p}).
\]
Here, we have $f(x)=x^r+s_1 x^{r-1}+\cdots+s_{r}$, and
suppose $s_0=1$ and $s_i=0$ for $i<0$.
For $(r,e,d)\in{\mathscr A}(p)$, if $e \ge p$, this equality implies that $x_1 \cdots x_r \mid \delta(x_1,\ldots,x_r)$.
This is a contradiction.
Hence, for $(r,e,d)\in{\mathscr A}(p)$, we should have $e \le p-1$.
\end{proof}
%

%
%
%
%
%
\begin{lemma}\label{UinB1}
\phantom{This is a phantom line.}
\begin{enumerate}
\item $\{ (r,e,d)\in{\mathscr U}(p)\ |\ d=r \} = {\mathscr B}_{+} (p)$ 
\item $\{ (r,e,d)\in{\mathscr U}(p)\ |\ d=r-1 \} = {\mathscr B}_0 (p)$
\item $\{ (r,e,d)\in{\mathscr U}(p)\ |\ r=2 \} = \{ (r,e,d)\in{\mathscr B}(p)\ |\ r=2\}$ 
\end{enumerate}
\end{lemma}
\begin{proof}
It is straightforward to verify these assertions.
\end{proof}
%

%
%
%
\begin{lemma}\label{UinB2}
\phantom{This is a phantom line.}
\begin{enumerate}
\item
$
\left\{ (r,e,d)\in{\mathscr U}(p) \left|\  p \mid r,\ \det M_{d}((x^r - x -1)^e) \ne 0\right.\right\}
=
\left\{ (r,e,d)\in{\mathscr B}(p) \left|\ p \mid r\right.\right\}
$
%
\item
$
\left\{ (r,e,d)\in{\mathscr U}(p) \left|\ p < r,\ p \nmid r,\ \det M_{d}((x^r-1)^e) \ne 0\right.\right\}
=
\left\{ (r,e,d)\in{\mathscr B}(p) \left|\ p < r,\ p \nmid r\right.\right\}
$
\item 
$
\left\{ (r,e,d)\in{\mathscr U}(p) \left|\ 2 \le r \le p-1,\ r \nmid p-1,\ \det M_{d}((x^r-1)^e) \ne 0\right.\right\}
$
\\
$=
\left\{ (r,e,d)\in{\mathscr B}(p) \left|\ 2 \le r \le p-1,\ r \nmid p-1\right.\right\}
$
\end{enumerate}
\end{lemma}
\begin{proof}
(i)
For $r \ge 2$ such that $p\mid r$, we put $f(x)=x^r-x-1$.
Then, by (\ref{Deltaxrx1}), we have $\Delta(f(x)) \ne 0$.
Hence, the right hand side of the assertion is contained in the left hand side.
Conversely, we take $(r,e,d)\in{\mathscr U}(p)$ such that $p \mid r$ and $\det M_{d}(f(x)^e) \ne 0$.
We put $r = \ell p$.
We have
\[
f(x)^e = \sum_{0 \le m \le e} (-1)^{e-m} \binom{e}{m} x^{m \ell p} (x+1)^{e-m}.
\]
By $(r,e,d)\in{\mathscr U}(p)$, we have $e \le p-1$.
Since we have $\det M_d(f(x)^e) \ne 0$, 
the rightmost column of $M_{d}(f(x)^e)$ is not a zero vector.
This implies $e=p-1$. 

First, we assume $\ell \ge 2$.
If $d\ge 2$, then the second row of $M_{d}(f(x)^e)$ is zero.
This contradicts $\det M_{d}(f(x)^e) \ne 0$.
Hence, we should have $d=1$.
Now, we have
\[
g = \frac{2}{\ell} \cdot \frac{p-1}{p} < 1.
\]
Hence that $\ell \ge 2$ is contradictory.

Second, we assume $\ell = 1$, i.e., $r=p$.
Then, we have
\[
g = 2d-\frac{d(d+1)}{p}.
\]
Hence, we obtain $d=p$ or $d=p-1$, i.e., $(r,e,d) = (p,p-1,p)$ or $(p,p-1,p-1)$.
These are contained in ${\mathscr B}_{+}(p)$ and ${\mathscr B}_{0}(p)$ respectively.
\noindent
(ii)
For $r \ge 2$ such that $p \nmid r$, we put $f(x)=x^r-1$.
Then, by (\ref{Deltaxr1}), we have $\Delta(f(x)) \ne 0$.
Hence, the right hand side of the assertion is contained in the left hand side.
Conversely, we take $(r,e,d)\in{\mathscr U}(p)$ such that $p<r$, $p \nmid r$ and $\det M_{d}(f(x)^e) \ne 0$.
We have
\begin{equation}\label{expandxrminus1}
f(x)^e = \sum_{0 \le m \le e} (-1)^{e-m} \binom{e}{m} x^{m r}.
\end{equation}
If $d < p$, then the first row of $M_{d}(f(x)^e)$ is zero.
Hence, we should have $d=p$.
Since $\det M_{d}(f(x)^e)\ne 0$ and $M_{d}(f(x)^e)$ is a $p \times p$ matrix,
$M_{d}(f(x)^e)$ has at least $p$ non-zero components.
Hence, we have $e+1 \ge p$, i.e., $e=p-1$. 
Since all the exponents $0$, $r$, $2r$, \ldots , $er$ lie in $M_{d}(f(x)^e)$, we have
$re \le p^2-1$, i.e., $r=p+1$.
Thus, we obtain $(r,e,d)=(p+1,p-1,p) \in {\mathscr B}_0 (p)$.
\noindent
(iii)
By the same reason as above for $f(x)=x^r-1$, the right hand side of the assertion is contained in the left hand side.
Conversely, we take $(r,e,d)\in{\mathscr U}(p)$ such that $2 \le r \le p-1$, $r \nmid p-1$ and $\det M_{d}(f(x)^e)\ne 0$.
We have (\ref{expandxrminus1}).
Since, for $0 \le m,\ m' \le e$, we have
\[
m r = m' r \bmod p \Longrightarrow m = m' \bmod p \Longrightarrow m=m',
\]
each column of the matrix $M_{d}(f(x)^e)$ has at most one non-zero component. 
Since $\det M_{d}(f(x)^e) \ne 0$, each column and each row of $M_{d}(f(x)^e)$ contain exactly one non-zero component.
The matrix $M_{d}(f(x)^e)$ gives a permutation $\sigma$ and a permutation matrix $\widetilde{M}$.
Hence, for each $1 \le i \le d$, we have $m_i$ such that $0 \le m_i \le e$ and $ip-d \le m_i r \le ip-1$.
The permutation $\sigma$ of $\{1, \ldots, d\}$ is given by
\[
\sigma(i) = m_i r - (ip-d)+1\ \ \  (1 \le i \le d).
\]
We write as $f(x)^e=\sum_{i\ge 0} c_i x^i$.
Then, by $r \nmid p-1$, we have $c_{p-1}=0$.
By $\det M_{d}(f(x)^e) \neq 0$, we have $d \ge 2$.
We also have $\widetilde{M} \ne {\rm J}_d$.

First, we assume $\widetilde{M} = {\rm I}_d$.
Then, since $d \ge 2$, we have 
\[
r \mid p-d \text{\ \ and\ \ } r \mid 2p-d+1.
\]
These imply $r \mid d+1$. 
Hence, we obtain $d=r-1$.
By Lemma \ref{UinB1}, we obtain $(r,e,d)\in{\mathscr B}_0(p)$.

Second, we assume $\widetilde{M} \ne {\rm I}_d$.
We define $h$ and $k$ by
\[
h = p \bmod r,\ \ 0 \le h \le r-1,\ \ k = -p \bmod r,\ \ 0 \le k \le r-1.
\]
Then, we have
\[
h \ge 2,\ \ k \ge 1,\ \ \gcd(h,k)=1,\ \ h+k=r.
\]
Since we have
\[
\sigma(i+1) - \sigma(i) = (m_{i+1} - m_i) r - p = -p \pmod r,
\]
and
\[
|\sigma(i+1)-\sigma(i)| \le d-1 \le r-1,
\]
we obtain
\[
\sigma(i+1) - \sigma(i) = -h \text{ or } k.
\]
Thus, $\widetilde{M}$ is a PPM of type $(h,k)$ of size $d$.
Since we have $\widetilde{M} \ne {\rm J}_d, {\rm I}_d$ and $2 \le d \le r = h+k$, $h \ge 2$, by Proposition \ref{PPMclassification}, 
we obtain $d=r$ or $d=r-1$.
Hence, by Lemma \ref{UinB1}, we obtain $(r,e,d)\in{\mathscr B}(p)$.
\end{proof}

From Lemma \ref{UinB1} (iii) and Lemma \ref{UinB2}, Theorem 2 follows.
%
%
%
\subsection{Proof of Theorem 3}
%
%
\begin{lemma}\label{toRho1}
For $(r,e,d)\in{\mathscr U}(p)$ such that $r \mid p-1$, we have
\begin{align*}
\det M_d((x^r-1)^e) &= (-1)^{d(d-1)/2 + (r-1) g/2}\binom{e}{s}\binom{e}{2s}\cdots\binom{e}{ds}.
\end{align*}
Here, we put $rs=p-1$.
In particular, we have $\det M_d((x^r-1)^e) \ne 0$.
\end{lemma}
\begin{proof}
By $r|p-1$, we have $p \ge 3$.
The matrix $M_d((x^r-1)^e)$ is equal to the following counter diagonal matrix.
\begin{equation*}
M_d((x^r-1)^e) =
\begin{pmatrix}
{} & {} & {} & (-1)^{e-s}\binom{e}{s} \\
{} & {} & (-1)^{e-2s}\binom{e}{2s} \\
{} & \reflectbox{$\ddots$} \\
(-1)^{e-ds}\binom{e}{ds}
\end{pmatrix}
\end{equation*}
\end{proof}
%

%
In what follows, to obtain $\mathscr{A}(p)=\mathscr{B}(p)$ for $p<200000$, we apply the successive testing
\begin{equation}\label{SuccessiveTest}
\det M_d(f(x)^e) = \frac{\det M_d((x^r-1)^e)}{\Delta(x^r-1)^{g/2}} \Delta(f(x))^{g/2}
\end{equation} 
by the test polynomials as follows.
\begin{align*}
&f(x) = x^r - x  \\
&f(x) = x^r + x^k + 1  \ (r > k > 0) \\
&f(x) = x^r + x^k + x  \ (r > k > 1) \\
&f(x) = x^r + a x + b  \ (a=1,\ b=2,\ 3)
\end{align*}
%
%

%
First, we shall consider the testing by $f(x)=x^r-x$.
We remark that 
\begin{equation}
\Delta(x^r-x)=(-1)^{(r+1)(r+2)/2} (r-1)^{r-1}, \label{Deltaxrx}
\end{equation}
which is non-zero.
Suppose we have $(r,e,d)\in{\mathscr U}(p)$ such that $r \mid p-1$ and $\det M_d((x^r-x)^e) \ne 0$.
By $r\mid p-1$, we have $p \ge 3$.
We have
\begin{equation*}
(x^r-x)^e = \sum_{0 \le m \le e} (-1)^{e-m} \binom{e}{m} x^{e+m(r-1)}.
\end{equation*}
Since, for $0 \le m,\ m' \le e$, we have
\[
e + m(r-1) = e + m'(r-1) \bmod p \Longrightarrow m = m' \bmod p \Longrightarrow m=m',
\]
each column of the matrix $M_d((x^r-x)^e)$ has at most one non-zero component. 
Since we assumed $\det M_d((x^r-x)^e) \ne 0$, each column and each row of $M_d((x^r-x)^e)$ contain exactly one non-zero component.
Therefore, for $i$ ($1 \le i \le d$), we have $0\le m_i \le e$ such that $ip-d \le e+m_i (r-1) \le ip-1$.
The matrix $M_d((x^r-x)^e)$ gives a permutation $\sigma$ and a permutation matrix $\widetilde{M}$.
We have
\begin{equation}
\sigma(i) = e + m_i (r-1) - (ip-d) + 1\ \ \ (1\le i \le d). \label{sigmaixrx}
\end{equation}
We consider the disjoint union
\[
\{ (r,e,d) \in {\mathscr U}(p)\ |\ r \mid p-1,\ \det M_d((x^r-x)^e) \ne 0 \} = \bigcup\limits_{j=1}^{4} {\mathscr C}_j(p),
\]
where ${\mathscr C}_1(p)$, ${\mathscr C}_2(p)$, ${\mathscr C}_3(p)$, ${\mathscr C}_4(p)$ are defined by the following conditions respectively.
\begin{align*}
{\mathscr C}_1(p) & \qquad d=1 & \\
{\mathscr C}_2(p) & \qquad d\ge 2, \widetilde{M}={\rm J}_d &\\
{\mathscr C}_3(p) & \qquad d\ge 2, \widetilde{M}={\rm I}_d &\\
{\mathscr C}_4(p) & \qquad d\ge 3, \widetilde{M}\ne{\rm J}_d, \widetilde{M}\ne{\rm I}_d &
\end{align*}
%
%
\begin{lemma}\label{C1C2C3}
\begin{enumerate}
%
\item
We have
\begin{align*}
&{\mathscr C}_1(p) = \left\{
(r,e,d) \in \Z^3 \left|
\begin{array}{l}
2 \le r \le p-1,\\
r \mid p-1\ (\text{{\rm we put} $rs=p-1$}), \\
e = s + (r-1) \ell,\ 1 \le \ell \le s, \\
d = 1
\end{array} 
\right.
\right\}. &
\end{align*}
For $(r,e,d)\in{\mathscr C}_1(p)$, we have
\begin{align*}
& \det M_d((x^r-x)^e) = (-1)^{r\;g/2} \binom{e}{s-\ell} . &
\end{align*}
%

%
\item
We have
\begin{align*}
& {\mathscr C}_2(p) = \left\{
(r,e,d) \in \Z^3 \left|
\begin{array}{l}
2 \le r \le p-1,\\
r(r-1) \mid p-1\ (\text{{\rm we put} $r(r-1)t=p-1$}), \\
e = (r-1) \ell,\ dt \le \ell \le rt, \\
2 \le d \le r
\end{array} 
\right.
\right\}
.&
\end{align*}
For $(r,e,d)\in{\mathscr C}_2(p)$, we have
\begin{align*}
& \det M_d((x^r-x)^e) = (-1)^{r\;g/2+d(d-1)/2} {\displaystyle\prod_{1\le i \le d}}\binom{e}{rti-\ell}.& 
\end{align*}
%

%
\item
We have
\begin{align*}
&{\mathscr C}_3(p) = \left\{
(r,e,d) \in \Z^3 \left|
\begin{array}{l}
3 \le r \le p-1, \\
r \mid p-1\ (\text{\rm{we put }} rs=p-1),\\
r-1 \mid p+1\ (\text{\rm{we put }} (r-1)t=p+1), \\
e = (r-1)\ell - (d+1),\ d(t-s) \le \ell \le t, \\
2 \le d \le r-1
\end{array} 
\right.
\right\}. &
\end{align*}
For $(r,e,d)\in{\mathscr C}_3(p)$, we have
\begin{align*}
&\det M_d((x^r-x)^e) = (-1)^{r\:g/2}{\displaystyle\prod_{1\le i \le d}}\binom{e}{ti-\ell} . &
\end{align*}
\end{enumerate}
\end{lemma}
\begin{proof}
%
%
(i) Suppose $(r,e,d)\in{\mathscr C}_1(p)$.
We put $rs=p-1$.
Then, we have $g=2(e-s)/(r-1)$.
Since $g\in2\Z$, we put $e=s+(r-1)\ell$ ($\ell\in\Z$).
By $g>0$, we obtain $\ell \ge 1$.
By $e \le p-1$, we obtain $\ell \le s$.
By $e+m_1 (r-1)=p-1$, we have $m_1=s-\ell$.
Then, the assertion follows.
\noindent (ii)
Suppose $(r,e,d)\in{\mathscr C}_2(p)$.
We have
\begin{align*}
e + m_1 (r-1) &= p-1, \\
e + m_2 (r-1) &= 2p-2.
\end{align*}
By these, we have $r-1 \mid p-1$ and $r-1 \mid e$.
We put $r(r-1)t=p-1$ and $e=(r-1)\ell$ ($\ell\in\Z$).
By $re \ge d(p-1)$, we obtain $dt \le \ell$.
By $e\le p-1$, we obtain $\ell\le rt$.
We have $m_i = rti-\ell$.
Then, the assertion follows.
\noindent (iii)
Suppose $(r,e,d)\in{\mathscr C}_3(p)$.
We put $rs=p-1$.
We have
\begin{align*}
e + m_1 (r-1) &= p-d, \\
e + m_2 (r-1) &= 2p-d+1.
\end{align*}
By these, we have $r-1 \mid p+1$ and $r-1 \mid e+d+1$.
We put $(r-1)t=p+1$ and $e=(r-1)\ell-(d+1)$ ($\ell\in\Z$).
We remark that $(r-1)(t-s)=s+2$.

We claim that $d \le r-1$.
If $d=r$, we have $e=p-1$.
This implies $m_1=-1$.
Hence, we should have $d \le r-1$.
In particular, we have $r \ge 3$.

By $e\ge 1$, we obtain $\ell \ge (d+2)/(r-1)$, i.e., $\ell \ge d(t-s)-(d s+d-2)/(r-1)$.
By $d(p-1) \le re$, we obtain $\ell \ge (d s+d+1)/(r-1)$, i.e., $\ell \ge d(t-s)-(d-1)/(r-1)$.
By $e\le p-1$, we obtain $\ell \le (r s+d+1)/(r-1)$, i.e., $\ell \le t+(d-1)/(r-1)$.
We have $g=2 d \ell - d(d+1)(t-s)$.
By $g>0$, we obtain $\ell > (d+1)(t-s)/2$, i.e., $\ell > d(t-s)-(d-1)(t-s)/2$.
By $m_1 \ge 0$, we obtain $\ell \le t$.
By $m_d \le e$, we obtain $\ell \ge (d t+d+1)/r$, i.e., $\ell \ge d(t-s)-(d-1)/r$.
By all of these, we obtain $d(t-s) \le \ell \le t$.
Then, the assertion follows.
\end{proof}
%
%
%
%
%
\begin{lemma}\label{Updeqrm2C4}
\begin{enumerate}
%
\item
\begin{equation*}
\{ (r,e,d) \in {\mathscr U}(p)\ |\ d=r-2 \}
=
\left\{
(r,e,d)\in\Z^3 \left|
\begin{array}{l}
r \ge 3,\ \ r\; |\; p-1\ \ (\text{{\rm we put} $rs=p-1$}), \\
e = (r-1)(s+\ell),\ -\lfloor \frac{s}{r-1} \rfloor + \delta(r,3) \le \ell \le \lfloor \frac{s}{r-1} \rfloor, \\
d = r-2
\end{array}
\right.
\right\}
\end{equation*}

Here, $\lfloor x \rfloor$ denotes the largest integer $n$ such that $n \le x$, and we put $\delta(r,3) = 1$ when $r=3$ and $\delta(r,3)=0$ when $r \ne 3$.

%
\item
For $(r,e,d)\in{\mathscr U}(p)$ such that $d=r-2$, we have
\[
\det M_d((x^r-x)^e) =  (-1)^{r\;g/2} \left[ \frac{-p}{\;r-1\;} \right] \prod_{1 \le i \le d} \binom{e}{m_i},
\]
where we put $m_i = \left\lceil \dfrac{ip-d}{r-1} \right\rceil - s - \ell$.
Here, $\lceil x \rceil$ denotes the least integer $n$ such that $x \le n$. 
In particular, we have $\det M_d((x^r-x)^e) \ne 0$.

%
\item
\begin{equation*}
{\mathscr C}_4(p) = \left\{ (r,e,d) \in {\mathscr U}(p)\ 
\left|
\begin{array}{l}
r \mid p-1,\ \ \ r-2 \le d \le r,\\
(r,e,d)\notin{\mathscr C}_1(p) \cup {\mathscr C}_2(p) \cup {\mathscr C}_3(p) 
\end{array}
\right\}
\right.
\end{equation*}
\end{enumerate}
\end{lemma}
\begin{proof}
%
%
(i) When $p=2$, both sides are the empty set.
We assume $p \ne 2$.
By $d=r-2$, we have
\[
g = 2e-(p-1)-\frac{2e}{r-1}+\frac{2(p-1)}{r}.
\]
By $g \in 2\Z$, we have $\frac{e}{r-1}$, $\frac{p-1}{r}\in\Z$.
We put $p-1=rs$ and $e=(r-1)(s+\ell)$ ($\ell\in\Z$).
By the condition $d(p-1)\le re \le r(p-1)$ of ${\mathscr U}(p)$, we obtain $ - \lfloor \frac{s}{r-1} \rfloor \le \ell \le \lfloor \frac{s}{r-1} \rfloor$.
Since $g=(r-2)(s+2\ell)$, we obtain $s+2\ell>0$.
Now the assertion follows.

\noindent (ii)
For $1 \le i \le d$, there exists a unique integer $m_i$ such that $ip-d \le e+m_i (r-1) \le ip$.
Then, we have $m_i = \lceil \frac{ip-d}{r-1} \rceil - s - \ell$.
We claim that $e+m_i(r-1)=(r-1)(s+\ell+m_i) \ne ip$.
If $(r-1)(s+\ell+m_i) = ip$, since $r \le p-1$, we have $p|e+\ell+m_i$.
Hence, we have $r-1 \le (r-1)(s+\ell+m_i)/p =i \le d = r-2$.
This is a contradiction.
%
Thus, we have
\[
ip-d \le e+m_i (r-1) \le ip-1.
\]
We have
\begin{align*}
& 0 \le m_1 \Leftrightarrow s+\ell \le \left\lceil \frac{p-d}{r-1} \right\rceil \Leftrightarrow s+\ell-1<\frac{p-d}{r-1} \Leftrightarrow e < p+1,\\
& m_d \le e \Leftrightarrow \left\lceil \frac{d(p-1)}{r-1} \right\rceil \le
r(s+\ell) \Leftrightarrow \frac{d(p-1)}{r-1} \le r(s+\ell)
\Leftrightarrow d(p-1) \le re.  \end{align*}
Thus, we obtain $0 \le m_1 < m_2 < \cdots < m_d \le e$.
We have
\begin{equation*}
e+m_{i_1}(r-1)  = e+m_{i_2}(r-1)  \bmod p \ \ \Rightarrow\ \  m_{i_1} = m_{i_2} \ \ \Rightarrow\ \  i_1 = i_2.
\end{equation*}
Hence, $M_d((x^r-x)^e)$ gives a permutation matrix $\widetilde{M}$ and a permutation $\sigma$.
We have
\[
\sigma(i) = (s+\ell+m_i)(r-1) - (ip-d) + 1\ \ \ (1\le i \le d).
\]
and
\[
\sigma(i+1) - \sigma(i) = (m_{i+1} - m_i)(r-1) -p \ \ \ (1\le i \le d-1).
\]
We define $h$ and $k$ by
\[
h = p \bmod r-1,\ \ 0 \le h \le r-2,\ \ k = -p \bmod r-1,\ \ 0 \le k \le r-2.
\]
Then, we have
\[
h \ge 1,\ \ k \ge 1,\ \ \gcd(h,k)=1,\ \ h+k=r-1.
\]
We have
\[
|\sigma(i+1)-\sigma(i)| \le d-1 = r-3.
\]
Hence, $\widetilde{M}$ is a PPM of type $(h,k)$ of size $h+k-1$.
By Proposition \ref{PPMclassification} and Proposition \ref{PPMdet}, we obtain 
\[
\widetilde{M} = K(h,k),\ \ \det\widetilde{M} =
\left[ \frac{-p}{\;r-1\;} \right].
\]
Since $\widetilde{M}$ is a permutation matrix, we have
\[
\sum_{1 \le i \le d} \{e+m_i (r-1)\}  = \sum_{1 \le i \le d} (ip-d) + \sum_{1 \le j \le d} (j-1) = \frac{d(d+1)}{2} (p-1).
\]
Hence,
\[
\sum_{1 \le i \le d} m_i = \frac{1}{r-1} \left\{ -e d + \frac{d(d+1)}{2} (p-1) \right\}.
\]
Thus, the assertion follows.
\noindent (iii)
Suppose that $(r,e,d)$ is contained in the right hand side set of the equality.
When $d=r-1$ or $r$, by Lemma \ref{UinB1}, we have $(r,e,d)\in{\mathscr B}(p)$.
Hence, $\det M_d((x^r-x)^e) \ne 0$.
When $d=r-2$, by (ii), we have $\det M_d((x^r-x)^e) \ne 0$.
Thus, we obtain $(r,e,d)\in{\mathscr C}_4(p)$.

Conversely suppose $(r,e,d)\in{\mathscr C}_4(p)$.
Let $\sigma$ and $\widetilde{M}$ be the permutation and the permutation matrix of $M_d((x^r-x)^e)$.
Then, by (\ref{sigmaixrx}), we have
\[
\sigma(i+1) - \sigma(i) = (m_{i+1} - m_i)(r-1) -p \ \ \ (1\le i \le d-1).
\]
We define $h$ and $k$ by
\[
h = p \bmod r-1,\ \ 0 \le h \le r-2,\ \ k = -p \bmod r-1,\ \ 0 \le k \le r-2.
\]
Then, we have
\[
h \ge 1,\ \ k \ge 1,\ \ \gcd(h,k)=1,\ \ h+k=r-1.
\]
Now suppose that $d \le r-3$.
Then, by
\[
|\sigma(i+1)-\sigma(i)| \le d-1 \le r-4 \le r-2,
\]
$\widetilde{M}$ is a PPM of type $(h,k)$ of size $d$.
By $3 \le d \le h+k-2$ and $\widetilde{M}\ne {\rm J}_d$, ${\rm I}_d$, Proposition \ref{PPMclassification} implies that such $\widetilde{M}$ does not exist.
This completes the proof.
\end{proof}
%

%
The following table is an excerpt from the computational output for Theorem 3.
The field (C$j$) ($1 \le j \le 4$) is the number of $(r,e,d)$ in the difference set ${\mathscr C}_j(p) \setminus {\mathscr B}(p)$.
To generate $(r,e,d)\in{\mathscr C}_j(p) \setminus {\mathscr B}(p)$, we apply Lemma \ref{C1C2C3} and Lemma \ref{Updeqrm2C4}.
The fields (T1), (T2), (T3), (T4) are the numbers of $(r,e,d)\in\cup_{j=1}^{4}{\mathscr C}_j(p) \setminus {\mathscr B}(p)$, which survives the successive testing (\ref{SuccessiveTest}) by
$f(x) = x^r-x$, $x^r + x^k + 1$ ($r > k > 0$), $x^r + x^k + x$ ($r > k > 1$), $x^r + a x + b$ ($a=1,\ b=2,\ 3$) respectively. 
For (T1), $\Delta(x^r-x)$ and $\det M_{d}((x^r-x)^e)$ are computed by (\ref{Deltaxrx}) and Lemma \ref{C1C2C3}, Lemma \ref{Updeqrm2C4}.
For (T2), (T3), (T4), $\Delta(f(x))$ and $\det M_{d}(f(x)^e)$ were computed directly in the environment of SageMath \cite{SM}.
The first prime for which (T2) is not zero is $p=193$.
The first prime for which (T3) is not zero is $p=6301$.
The field (T4) is zero for all $p<200000$.
Hence, the proof of Theorem 3 is complete.
The primes $p=199523$, $199679$ are examples of primes such that $(p-1)/2$ are primes.
These are what we shall consider in Theorem 4. 
%
%
\setcounter{table}{-1}
\begin{center}
\begin{longtable}{|r|rrrr|rrrr|}
\hline
\multicolumn{1}{|c|}{{$p$}} & {(C1)} & {(C2)} & {(C3)} & {(C4)} & {(T1)} & {(T2)} & {(T3)} & {(T4)} \\
\hline
\endfirsthead
\hline
\multicolumn{1}{|c|}{{$p$}} & {(C1)} & {(C2)} & {(C3)} & {(C4)} & {(T1)} & {(T2)} & {(T3)} & {(T4)} \\
\hline
\endhead
\hline
\endfoot
\hline
\endlastfoot
3 &   0  & 0  & 0  & 0  & 0  & 0 & 0 & 0 \\
5 &   1  & 0  & 0  & 0  & 0  & 0 & 0 & 0 \\
7 &   2  & 0  & 0  & 0  & 0  & 0 & 0 & 0 \\
11 &  3  & 0  & 2  & 0  & 0  & 0 & 0 & 0 \\
13 &  9  & 2  & 0  & 0  & 1  & 0 & 0 & 0 \\
17 &  7  & 0  & 2  & 0  & 0  & 0 & 0 & 0 \\
19 &  11 & 0  & 5  & 0  & 0  & 0 & 0 & 0 \\
23 &  3  & 0  & 0  & 0  & 0  & 0 & 0 & 0 \\
29 &  14 & 0  & 13 & 0  & 1  & 0 & 0 & 0 \\
31 &  26 & 11 & 7  & 0  & 10 & 0 & 0 & 0 \\
37 &  36 & 6  & 0  & 2  & 1  & 0 & 0 & 0 \\
41 &  30 & 11 & 20 & 0  & 1  & 0 & 0 & 0 \\
43 &  32 & 17 & 0  & 2  & 9  & 0 & 0 & 0 \\
$\vdots$ & $\vdots$ & $\vdots$ & $\vdots$ & $\vdots$ & $\vdots$ & $\vdots$ & $\vdots$ & $\vdots$ \\
193 &  219 & 32 & 0 & 20   & 4 & 1 & 0 & 0 \\
$\vdots$ & $\vdots$ & $\vdots$ & $\vdots$ & $\vdots$ & $\vdots$ & $\vdots$ & $\vdots$ & $\vdots$ \\
6301 &  13117 & 17642 & 0 & 492   & 8 & 1 & 1 & 0 \\
$\vdots$ & $\vdots$ & $\vdots$ & $\vdots$ & $\vdots$ & $\vdots$ & $\vdots$ & $\vdots$ & $\vdots$ \\
199523 & 3 & 0 & 0 & 0 & 0 & 0 & 0 & 0 \\
$\vdots$ & $\vdots$ & $\vdots$ & $\vdots$ & $\vdots$ & $\vdots$ & $\vdots$ & $\vdots$ & $\vdots$ \\
199679 & 3 & 0 & 0 & 0 & 0 & 0 & 0 & 0 \\
$\vdots$ & $\vdots$ & $\vdots$ & $\vdots$ & $\vdots$ & $\vdots$ & $\vdots$ & $\vdots$ & $\vdots$ \\
199873 & 274739 & 108269 & 0 & 28330 & 8 & 0 & 0 & 0 \\
199877 & 53994 & 0 & 33312 & 44 & 0 & 0 & 0 & 0 \\
199889 & 131130 & 0 & 213213 & 10646 & 4 & 0 & 0 & 0 \\
199909 & 223583 & 33318 & 59972 & 11018 & 1 & 0 & 0 & 0 \\
199921 & 463463 & 1017128 & 0 & 14464 & 23 & 0 & 0 & 0 \\
199931 & 59997 & 0 & 49982 & 4442 & 1 & 0 & 0 & 0 \\
199933 & 166637 & 33322 & 0 & 16356 & 1 & 0 & 0 & 0 \\
199961 & 150060 & 49991 & 182502 & 1308 & 3 & 0 & 0 & 0 \\
199967 & 23115 & 0 & 83319 & 614 & 0 & 0 & 0 & 0 \\
199999 & 145538 & 0 & 229997 & 1404 & 2 & 0 & 0 & 0 \\
\hline
\end{longtable}
\renewcommand{\arraystretch}{1}
\captionof{table}{Table for Theorem 3}\label{TableTh3}
\end{center}
\setcounter{table}{-0}
\begin{center}
\renewcommand{\arraystretch}{1.1}
\begin{longtable}{|c|rrrr|}
\hline
{} & (T1) & (T2) & (T3) & (T4) \\
\hline
Average & 5.52694 & 0.01006 & 0.00266 & 0 \\
Maximum & 1170 & 2 & 1 & 0 \\
\hline
\end{longtable}
\renewcommand{\arraystretch}{1}
\captionof{table}{Statistical Summary of Table for Theorem 3}\label{TableTh3AverageMaximum}
\end{center}
%
%
%
%
%
%
\subsection{Proof of Theorem 4}
%

%
%
Suppose a prime $p$ and $(r,e,d)\in{\mathscr C}_1(p)$ are given.
By Lemma \ref{C1C2C3} (i), we have $rs=p-1$, $e=s+(r-1)\ell$ ($1\le \ell \le s$), $d=1$.
Then, by Lemma \ref{toRho1}, (\ref{Deltaxr1}), Lemma \ref{C1C2C3} (i), (\ref{Deltaxrx}), we obtain
\begin{align}
& \frac{\det M_d((x^r-1)^e)}{\Delta(x^r-1)^{g/2}} = \frac{\det M_d((x^r-x)^e)}{\Delta(x^r-x)^{g/2}} \label{C1Survive} \\
\Longleftrightarrow\ \ \ 
& \frac{(-1)^{(r-1)\ell}\binom{e}{s}}{\left\{(-1)^{(r-1)(r-2)/2}r^r\right\}^{\ell}} 
= \frac{(-1)^{r\ell}\binom{e}{s-\ell}}{\left\{(-1)^{(r+1)(r+2)/2} (r-1)^{r-1} \right\}^{\ell}} \notag \\
\Longleftrightarrow\ \ \  
& \left( \frac{-1}{s+1}\right)^{\ell} \frac{\ell (\ell + s) \cdots (\ell + (\ell-1)s)}{s (s-1) \cdots (s-\ell+1)} = \left( \frac{-1}{s+1}\right)^{r \ell}. \notag
\end{align}
Here, we remark that
\begin{align*}
\frac{\binom{e}{s}}{\binom{e}{s-\ell}} &=  \frac{(e-s+\ell)(e-s+\ell-1)\cdots(e-s+1)}{(s-(\ell-1))(s-(\ell-2))\cdots s} \\
&= \frac{r\ell (r\ell-1)\cdots(r\ell-(\ell-1))}{(s-(\ell-1))(s-(\ell-2))\cdots s} \\
&= \frac{1}{s^{\ell}}\frac{rs\ell (rs\ell-s)\cdots(rs\ell-(\ell-1)s)}{(s-(\ell-1))(s-(\ell-2))\cdots s} \\
&= \frac{1}{s^{\ell}}\frac{(-\ell) (-\ell-s)\cdots(-\ell-(\ell-1)s)}{(s-(\ell-1))(s-(\ell-2))\cdots s} \\
&= \frac{(-1)^{\ell}}{s^{\ell}}\frac{\ell (\ell+s)\cdots(\ell+(\ell-1)s)}{ s (s-1) \cdots (s-\ell+1)},
\end{align*}
and
\begin{equation*}
\frac{r^r}{(r-1)^{r-1}} = \frac{(rs)^r}{(rs-s)^{r-1} s} = \frac{(-1)^r}{(-1-s)^{r-1} s} = \frac{-1}{(s+1)^{r-1} s}.
\end{equation*}

Now, for integers $s$ and $\ell$ such that $1 \le \ell \le s$, we define a rational number $\kappa(s,\ell)$ by
\begin{equation*}
\kappa(s,\ell) = \left( \frac{-1}{s+1}\right)^{\ell} \frac{\ell (\ell + s) \cdots (\ell + (\ell-1)s)}{s (s-1) \cdots (s-\ell+1)}. 
\end{equation*}
Then, since $\dfrac{1}{s+1} \dfrac{\ell+ks}{s-\ell+1+k} < 1$ for $0 \le k \le \ell-1$, we have $0 < |\kappa(s,\ell)| < 1$.
Also, any prime divisor of the denominator of $\kappa(s,\ell)$ is less than or equal to $s+1$.
By the above, for $(r,e,d)\in{\mathscr C}_1(p)$, we obtain 
\begin{align}
(\ref{C1Survive}) \ \ \ 
\Longleftrightarrow\ \ \ 
& \kappa(s,\ell) = \left(\frac{-1}{s+1} \right)^{r\ell} \text{ in $\F_p$} \label{C1SurviveCond1}\\
\Longrightarrow\ \ \ 
& \kappa(s,\ell)^s = 1 \text{ in $\F_p$} \label{C1SurviveCond2}.
\end{align}
When $s$ and $\ell$ are fixed, by applying (\ref{C1SurviveCond2}) first and (\ref{C1SurviveCond1}) second with the conditions $s \mid p-1$ and $\frac{p-1}{s} \ge 2$,
the set of primes $p$, for which $(r,e,d)$ survives the condition (\ref{C1Survive}) can be effectively computed.
The following is the table for $s \le 3$ of $p$ and $(r,e,d)$ satisfying (\ref{C1Survive}).
%
%
%
%
\setcounter{table}{0}
\begin{center}
\renewcommand{\arraystretch}{1.2}
\begin{longtable}{|c|c|c|r|l|c|}
\hline
$s$ & $\ell$ & $\kappa(s,\ell)$ & $p$ & $(r,e,d)$ & in ${\mathscr B}(p)$ or not\\
\hline
1 & 1 & $-\frac{1}{2}$   & 3  & $(2,2,1)$   & $\in{\mathscr B}_0(p)$ \\
\hline
2 & 1 & $-\frac{1}{6}$   & 5  & $(2,3,1)$   & $\in{\mathscr B}_0(p)$ \\
\cline{4-6}
  &   &                  & 7  & $(3,4,1)$   & $\in{\mathscr B}_-(p)$ \\
\cline{2-6}
  & 2 & $\frac{4}{9}$    & 5  & $(2,4,1)$   & $\in{\mathscr B}_0(p)$ \\
\hline
3 & 1 & $-\frac{1}{12}$  & 7  & $(2,4,1)$   & $\in{\mathscr B}_0(p)$ \\ 
\cline{2-6}
  & 2 & $\frac{5}{48}$   & 7  & $(2,5,1)$   & $\in{\mathscr B}_0(p)$ \\
\cline{4-6}
  &   &                  & 43 & $(14,29,1)$ & $\notin{\mathscr B}(p)$ \\ 
\cline{2-6}
  & 3 & $-\frac{27}{64}$ & 7  & $(2,6,1)$   & $\in{\mathscr B}_0(p)$ \\ 
\cline{4-6}
  &   &                  & 13 & $(4,12,1)$  & $\notin{\mathscr B}(p)$ \\
\hline
\end{longtable}
\renewcommand{\arraystretch}{1}
\captionof{table}{$\kappa(s,\ell)$}\label{TableKappa}
\end{center}
%
%
%
%

%
%
Now we shall give a proof of Theorem 4.
We put $p=2\pi+1$ with a prime $\pi$.
We let the part of $p \le 11$ of Table 1 be admitted.
Hence, we assume $\pi > 5$.

Suppose we have $(r,e,d)\in{\mathscr C}_1(p)$.
By $r \mid p-1=2\pi$, we have $r=2$ or $\pi$ or $2\pi$.
If $r=2$, by Lemma \ref{UinB1} (iii), we have $(r,e,d)\in{\mathscr B}(p)$.
If $r=\pi$ or $2\pi$, we have $s=2$ or $1$ accordingly.
By Table 2, since $p>11$, we have $(r,e,d)\notin{\mathscr A}(p)$.
Thus we obtain $({\mathscr C}_1(p) \setminus {\mathscr B}(p))\cap{\mathscr A}(p) = \emptyset$.

Suppose we have $(r,e,d)\in{\mathscr C}_2(p)$.
By Lemma \ref{C1C2C3} (ii), we have $r(r-1) \mid p-1$.
If $r=2$, by Lemma \ref{UinB1} (iii), we have $(r,e,d)\in{\mathscr B}(p)$.
If $r\ge 3$, then $r(r-1) \mid 2\pi$ implies $\pi=3$.
Thus we obtain ${\mathscr C}_2(p) \subset {\mathscr B}(p)$.

Suppose we have $(r,e,d)\in{\mathscr C}_3(p)$.
By Lemma \ref{C1C2C3} (iii), we have $3 \le r$, $r \mid p-1$, $r-1\mid p+1$.
Then, these conditions imply $\pi \le 5$.
Thus we obtain ${\mathscr C}_3(p) = \emptyset$.

Suppose we have $(r,e,d)\in{\mathscr C}_4(p)$.
By Lemma \ref{Updeqrm2C4} (iii), we have $r-2 \le d \le r$.
If $r-1 \le d \le r$, by Lemma \ref{UinB1} (i), (ii), we have $(r,e,d)\in{\mathscr B}(p)$.
If $d=r-2$, by Lemma \ref{Updeqrm2C4} (i), we have $e=(r-1)(s+\ell)$ ($-\lfloor \frac{s}{r-1}\rfloor +\delta(r,3)\le \ell \le \lfloor \frac{s}{r-1} \rfloor$).
If $\ell=0$, we have $(r,e,d)\in{\mathscr B}_{-}(p)$.
If $\ell\ne 0$, we have $\lfloor \frac{s}{r-1} \rfloor \ge 1$.
Hence $r(r-1) \le p-1$.
This, together with $r \ge 3$ and $r \mid p-1$, implies $\pi \le 3$.
Thus we obtain ${\mathscr C}_4(p) \subset {\mathscr B}(p)$.

These complete the proof of Theorem 4.
%
%
\section{Proof of $M_d(f(x)^e)^{-1}M_d(f(x)^{e+1})^{}=B_r \ Q_r \ Z_{r,n} \ P_r$}
In this Section, we shall prove Theorem 5.
%
%
\subsection{}
Let $p$ be a prime and $(r,e,d)\in{\mathscr B}_0(p)$ be such that $(r,e+1,d)\in{\mathscr B}_0(p)$.
We remark that this condition implies $p\ge 5$ and $r \le p-1$.
Note that $d=r-1$.
We define $n$ by $p-1=e+n$.
We take independent variables $s_1$, \ldots, $s_r$ over $\F_p$ and put $f(x)=x^r+s_1 x^{r-1}+\cdots+s_r$.
We suppose $s_0=1$ and $s_i = 0$ for $i<0$ or $i>r$.
In general, for a Laurent series $\Phi(x)=\sum_{i\in\Z} C_i x^i$, we denote $C_i = [x^i]\Phi(x)$ for $i\in\Z$.
We remark that
\begin{equation}\label{OriginRelation}
[t^k] t \Phi'(t) = [t^k] k \Phi(t)\ \ \ (k\in\Z).
\end{equation}
%

%
We put $c_i = [x^i]f(x)^{e}$ ($i\in\Z$).
Let $L$ be the $d \times (r+d)$ matrix given by
\begin{equation}\label{matL}
L = \begin{pmatrix*}[l] 
c_{p-2r+1}  & \cdots & c_{p-r}  & c_{p-d}  & \cdots & c_{p-1} \\
c_{2p-2r+1} & \cdots & c_{2p-r} & c_{2p-d} & \cdots & c_{2p-1} \\
\vdots      &        & \vdots   & \vdots   &        & \vdots \\
c_{dp-2r+1} & \cdots & c_{dp-r} & c_{dp-d} & \cdots & c_{dp-1}
\end{pmatrix*}.
\end{equation}
The rightmost $d \times d$ matrix in $L$ is equal to $M_{d}(f(x)^e)$.
Let $V$ be the $(r+d)\times d$ matrix given by
\begin{equation}\label{matV}
V = \begin{pmatrix*}[l] 
s_0     &        &        \\
s_1     & \ddots &        \\
\vdots  & \ddots & s_0    \\
s_{r-1} & \cdots & s_1    \\
s_r     & \cdots & s_2    \\
{}      & \ddots & \vdots \\
{}      & {}     & s_r    
\end{pmatrix*}.
\end{equation}
%
%
\begin{lemma} \label{LVeqMe1}
$\ \ \ L\ V = M_{d}(f(x)^{e+1})$.
\end{lemma}
\begin{proof}
%
%
%
%
%
%
By
\begin{align*}
L_{i,j} &= [x^{ip+j-2r}] f(x)^e \ \ \ (1 \le i \le d,\ 1 \le j \le r+d), \\
V_{i,j} &= [x^{r-i+j}] f(x) \ \ \ (1 \le i \le r+d,\ 1 \le j \le d), \\
M_{d}(f(x)^{e+1})_{i,j} &= [x^{ip+j-r}] f(x)^{e+1}\ \ \ (1 \le i \le d,\ 1 \le j \le d),
\end{align*}
we obtain the assertion.
%
%
%
%
%
%
\end{proof}
%

%
By (\ref{OriginRelation}), we have
\begin{equation}\label{theRelation0}
[x^{ip+j-r}] \left\{ -x \frac{d}{dx} f(x)^{e+1} + (ip+j-r) f(x)^{e+1}\right\} = 0\ ( 1 \le i \le d,\ 1 \le j \le r),
\end{equation}
which is equivalent to
\begin{equation}\label{theRelation1}
[x^{ip+j-r}] f(x)^e \{ n x f'(x) - (r-j) f(x) \} = 0 \ \ \ ( 1 \le i \le d,\ 1 \le j \le r).
\end{equation}
Let $R$ be the $(r+d)\times r$ matrix given by
%
%
%
%
\begin{align} \label{matR}
R &= 
\ n 
\begin{pmatrix}
r s_0     \\
(r-1) s_1 & r s_0     \\
\vdots    & \ddots    & \ddots    \\
\vdots    &           & \ddots    & \ddots    \\
1 s_{r-1} & \cdots    & \cdots    & (r-1) s_1 & r s_0     \\
0 s_r     & 1 s_{r-1} & \cdots    & \cdots    & (r-1) s_1 \\
          & 0 s_r     & 1 s_{r-1} &           & \vdots    \\
          &           & \ddots    & \ddots    & \vdots    \\
          &           &           & 0 s_r     & 1 s_{r-1}
\end{pmatrix}
\\
&- \nonumber
\begin{pmatrix}
s_0     \\
s_1     & s_0     \\
\vdots  & \ddots  & \ddots  \\
\vdots  &         & \ddots  & \ddots \\
s_{r-1} & \cdots  & \cdots  & s_1    & s_0    \\
s_r     & s_{r-1} & \cdots  & \cdots & s_1    \\
        & s_r     & s_{r-1} &        & \vdots \\
        &         & \ddots  & \ddots & \vdots \\
        &         &         & s_r    & s_{r-1}
\end{pmatrix}
\begin{pmatrix}
r-1 \\
    & r-2 \\
    &     & \ddots \\
    &     &        & 1 \\
    &     &        &   & 0 
\end{pmatrix}.
\end{align}
The matrix $R$ is of the form $R = n A + B$.
In such a case, we call $A$ and $B$ the $n$-part and the $1$-part of $R$ respectively. 
%
%
\begin{lemma}\label{LReq0}
$
\ \ \ L\ R = 0.
$
\end{lemma}
\begin{proof}
By
\begin{equation*}
R_{i,j} =  [x^{r-i+j}]\{ n x f'(x) - (r-j) f(x)\} \ \ \ (1\le i \le r+d,\ 1\le j \le r)
\end{equation*}
and (\ref{theRelation1}), we obtain the assertion.
\end{proof}
%

%
Let $R_1$ and $R_2$ be the upper $r \times r$ matrix and the lower $d \times r$ matrix in $R$ respectively.
\begin{equation} \label{RR1R2}
R = \begin{pmatrix} R_1 \\ R_2 \end{pmatrix}
\end{equation}
We put $L = \begin{pmatrix} L' & M_{d}(f(x)^e) \end{pmatrix}$.
Then, by Lemma \ref{LReq0}, we have
\[
L R =
\begin{pmatrix} L' & M_{d}(f(x)^e) \end{pmatrix}
\begin{pmatrix} R_1 \\ R_2 \end{pmatrix} 
= L' R_1 + M_{d}(f(x)^e) R_2 = 0,
\]
which implies $L' = - M_{d}(f(x)^e) R_2^{} R_1^{-1}$, hence
\[
L = M_{d}(f(x)^e) \begin{pmatrix} -R_2^{} R_1^{-1} & {\rm I}_d \end{pmatrix}. 
\]
By Lemma \ref{LVeqMe1}, we obtain
\[
\begin{pmatrix} -R_2^{} R_1^{-1} & {\rm I}_d \end{pmatrix} V = M_{d}(f(x)^e)^{-1} \ M_{d}(f(x)^{e+1}).
\]
Hence, it is sufficient to show
\begin{equation}\label{target1}
\begin{pmatrix} - R_2^{} R_1^{-1} & {\rm I}_d \end{pmatrix} V = B_r \ Q_r \ Z_{r,n} \ P_r
\end{equation}
to prove Theorem 5.
In what follows, we shall show (\ref{target1}) in characteristic $0$ assuming
$s_1$, \ldots, $s_r$, $n$ are independent variables over $\Q$ 
with matrices $R$ and $V$ defined over $\Q$ by (\ref{matR}) and (\ref{matV}) respectively.
%
%
\subsection{The matrices $P(\lambda_1, \ldots, \lambda_m)$, $Q(\mu_1, \ldots, \mu_m)$, $S_m (\psi(t))$ }
For an integer $r\ge 2$, let $s_1$, \ldots, $s_r$ be independent variables over $\Q$.
We put $\varphi(t)=1+s_1 t+\cdots+s_r t^r$.
For $\lambda \in \C$ and $\ell\in\Z$, we define $\beta_{\ell}(\lambda)$ by
\[
\beta_{\ell}(\lambda) = [t^{\ell}] \varphi(t)^{\lambda}.
\]
Note that $\beta_{\ell}(\lambda)$ can be written as
\[
\beta_{\ell}(\lambda)=
\sum_{\substack{k_1,\ldots,\;k_r \ge 0\\k_1+2k_2+\cdots+rk_r=\ell}} \binom{\lambda}{k_1+\cdots+k_r} \frac{(k_1+\cdots+k_r)!}{k_1!\cdots k_r!} s_1^{k_1}\cdots s_r^{k_r}.
\]
Here, $\binom{\lambda}{k_1+\cdots+k_r}$ denotes the binomial coefficient.
In particular, we have $\beta_0(\lambda)=1$ and $\beta_{\ell}(\lambda)=0$ for $\ell < 0$.
For $m \ge 1$ and $\lambda_1$, \ldots, $\lambda_m \in \C$, we define an $m\times m$ matrix $P(\lambda_1,\ldots,\lambda_m)$ by
\[
P(\lambda_1, \ldots, \lambda_m)_{i,j} = \beta_{i-j}(\lambda_i) \ \ \ (\ 1 \le i,\ j \le m\ ).
\]
Similarly, for $m \ge 1$ and $\mu_1$, \ldots, $\mu_m \in \C$, we define an $m\times m$ matrix $Q(\mu_1,\ldots,\mu_m)$ by
\[
Q(\mu_1, \ldots, \mu_m)_{i,j} = \beta_{i-j}({\mu_j}) \ \ \ (\ 1 \le i,\ j \le m\ ).
\]
Here, $\lambda_1$ and $\mu_m$ are dummy parameters.
For $m \ge 1$ and $\lambda \in \C$, we put
\[
U_m(\lambda) = P(\underbrace{\lambda,\ldots,\lambda}_{m}) = Q(\underbrace{\lambda,\ldots,\lambda}_{m}). 
\]
Then, these satisfy the following relations.
\begin{align}
& \{ P(\lambda_1,\ldots,\lambda_m)\ Q(\mu_1,\ldots,\mu_m) \}_{i,j} = \beta_{i-j}(\lambda_i + \mu_j)\label{PQcomponent} \\
& P(\lambda_1,\ldots,\lambda_m)\ U_m(\mu) = P(\lambda_1+\mu,\ldots,\lambda_m+\mu) \label{PU}\\
& U_m(\lambda)\ Q(\mu_1,\ldots,\mu_m) = Q(\lambda+\mu_1,\ldots,\lambda+\mu_m) \label{UQ}\\
& U_m(\lambda)\ U_m(\mu) = U_m(\lambda+\mu) \label{UU}\\
& U_m(0) = {\rm I}_m \label{U1}
\end{align}
The matrices $P_r$, $Q_r$ introduced in (\ref{matPr}), (\ref{matQr}) can be written as follows.
\begin{align*}
P_r &= P\left(-\frac{r-1}{r}, -\frac{r-2}{r}, \ldots , -\frac{1}{r}\right) \\ 
Q_r &= Q\left(-\frac{1}{r}, -\frac{2}{r}, \ldots, -\frac{r-1}{r}\right)
\end{align*}
For $m \ge 1$ and a power series $\psi(t)=a_0 + a_1 t + \cdots$, we put
\[
S_m (\psi(t)) = \begin{pmatrix}
a_0     &     &        &     &     &    \\
a_1     & a_0 &        &     &     &    \\
{}      & a_1 & \ddots &     &     &    \\
\vdots  & {}  & \ddots & \ddots &     &    \\
{}      & {}  & {}     & a_1 & a_0 & {} \\
a_{m-1} & {}  & \cdots & {}  & a_1 & a_0
\end{pmatrix},
\]
i.e., $S_m(\psi(t))_{i,j} = a_{i-j}$ assuming $a_i = 0$ for $i<0$.
Then, the followings are satisfied.
\begin{align}
& S_m(\psi_1(t) + \psi_2(t)) = S_m(\psi_1(t)) + S_m(\psi_2(t)) \label{Sadd}\\
& S_m(\psi_1(t)\ \psi_2(t)) = S_m(\psi_1(t))\ S_m(\psi_2(t)) \label{Smult}\\
& S_m( 1 ) = {\rm I}_m \label{S1}\\
& S_m(\varphi(t)^{\lambda}) = U_m(\lambda) \label{SU}
\end{align}
We remark that
\begin{equation}\label{R1bynS}
R_1= n\ S_r(r \varphi(t)-t \varphi'(t)) - S_r(\varphi(t)) \begin{pmatrix} r-1 \\ {} & r-2 \\ {} & {} & \ddots \\ {} & {} & {} & 1 \\ {} & {} & {} & {} & 0 \end{pmatrix}.
\end{equation}
%
%
\subsection{Some Fundamental Relations of Matrices}
We take $r \ge 2$ and
assume that $s_1$, \ldots , $s_r$, $n$ are independent variables over $\Q$.
We put $\varphi(t)=1+s_1 t+\cdots+s_r t^r$.
%
%
\begin{lemma} \label{lemmaUsingRes}
We assume $r \ge 2$ and $m \ge 1$.
\begin{enumerate}
\item 
\begin{equation*}
P\left( -\frac{r-1}{r}, -\frac{r-2}{r}, \ldots, -\frac{r-m}{r}\right)
S_m\left(r-\frac{t\, \varphi'(t)}{\varphi(t)}\right)
Q\left(\frac{r-1}{r}, \frac{r-2}{r}, \ldots, \frac{r-m}{r}\right)
= 
r \; {\rm I}_m
\end{equation*}
%
%
\item
\begin{align*}
&
P\left( -\frac{r-1}{r}, -\frac{r-2}{r}, \ldots, -\frac{r-m}{r}\right)
\begin{pmatrix}
r-1 \\
& r-2 \\
&& \ddots \\
&&& r-m
\end{pmatrix}
Q\left(\frac{r-1}{r}, \frac{r-2}{r}, \ldots, \frac{r-m}{r}\right)
\\
& 
= 
\begin{pmatrix}
r-1 \\
& r-2 \\
&& \ddots \\
&&& r-m
\end{pmatrix}
\end{align*}
%
%
%
\end{enumerate}
\end{lemma}
\begin{proof}
\noindent
(i)
We briefly denote the three matrices in the left hand side by $P$, $S$, $Q$.
We put $\psi(t)=r -\dfrac{t \varphi'(t)}{\varphi(t)}$.
For $1\le i,k,\ell,j \le m$, we have
\begin{align*}
P_{i,k} &= [t^{i-k}] \varphi(t)^{-(r-i)/r}, \\
S_{k,\ell} &= [t^{k-\ell}] \psi(t), \\
Q_{\ell,j} &= [t^{\ell - j}] \varphi(t)^{(r-j)/r}.
\end{align*}
By these, we obtain
\begin{align*}
(PS)_{i,\ell} &= \sum_{1\le k \le m} [t^{i-k}]\varphi(t)^{-(r-i)/r}\ [t^{k-\ell}] \psi(t) \\
&= [t^{i-\ell}] \varphi(t)^{-(r-i)/r} \psi(t) , \\
(PSQ)_{i,j} &= \sum_{1 \le \ell \le m} [t^{i-\ell}] \varphi(t)^{-(r-i)/r} \psi(t)\ [t^{\ell - j}]\varphi(t)^{(r-j)/r} \\
&= [t^{i-j}] \varphi(t)^{(i-j)/r} \psi(t) \\
&= \underset{t=0}{{\rm Residue}} \ \varphi(t)^{(i-j)/r} \psi(t) \frac{dt}{t^{i-j+1}} \\
&= \underset{u=0}{{\rm Residue}} \ r \frac{du}{u^{i-j+1}} \text{ (we put $u = \frac{t}{\varphi(t)^{1/r}}$)} \\
&= \left\{ \begin{array}{cc} 
r & ( i=j ) \\
0 & ( i\ne j)
\end{array}
\right.
.
\end{align*}
\noindent
(ii)
We denote the matrix in the right hand side by $J$.
Then, we obtain
\begin{align*}
(PJ)_{i,k} &= [t^{i-k}] \varphi(t)^{-(r-i)/r} (r-k) \\
&= [t^{i-k}] \{(i-k)+(r-i)\} \varphi(t)^{-(r-i)/r} \\
&= [t^{i-k}] \left\{t \frac{d}{dt} \varphi(t)^{-(r-i)/r} + (r-i) \varphi(t)^{-(r-i)/r}\right\} \\
&= [t^{i-k}] \frac{r-i}{r} \varphi(t)^{-(r-i)/r} \psi(t), \\
(PJQ)_{i,j} &=\sum_{1 \le k \le m} [t^{i-k}] \frac{r-i}{r} \varphi(t)^{-(r-i)/r} \psi(t)\ [t^{k-j}] \varphi(t)^{(r-j)/r} \\
&= [t^{i-j}] \frac{r-i}{r} \varphi(t)^{(i-j)/r} \psi(t) \\
&= \underset{t=0}{{\rm Residue}}\ \frac{r-i}{r} \varphi(t)^{(i-j)/r} \psi(t) \frac{dt}{t^{i-j+1}} \\
&= \underset{u=0}{{\rm Residue}}\ (r-i) \frac{du}{u^{i-j+1}} \\
&= \left\{ \begin{array}{lc} 
r-i & ( i=j ) \\
0 & ( i\ne j)
\end{array}
\right.
.
\end{align*}
\end{proof}

%
Let $R_1$ be the matrix given by (\ref{RR1R2}) and (\ref{R1bynS}).
We put $\zeta_i = \dfrac{n}{rn - (r-i)}$ for $1 \le i \le r$.
Note that $\zeta_r=\frac{1}{r}$.
%
%
\begin{lemma} \label{lemmaR1inv}
\begin{equation}\label{R1inv}
R_1^{-1} = 
\frac{1}{n}\ 
Q\left(\frac{r-1}{r},\frac{r-2}{r},\ldots,\frac{1}{r},\frac{0}{r}\right)
\begin{pmatrix}
\zeta_1 &         &         &         \\
        & \zeta_2 &         &         \\
        &         & \ddots  &         \\
        &         &         & \zeta_r
\end{pmatrix}
P\left(-\frac{2r-1}{r},-\frac{2r-2}{r},\ldots,-\frac{r+1}{r},-\frac{r}{r}\right)
\end{equation}
\end{lemma}
\begin{proof}
By (\ref{R1bynS}), the equality (\ref{R1inv}) is equivalent to
\begin{equation*}
\begin{split}
&
P\left(-\frac{2r-1}{r},-\frac{2r-2}{r},\ldots,-\frac{r+1}{r},-\frac{r}{r}\right)
\\ 
&
\times 
\left\{
n\ S_r(r \varphi(t)-t \varphi'(t)) - S_r(\varphi(t)) \begin{pmatrix} r-1 \\ {} & r-2 \\ {} & {} & \ddots \\ {} & {} & {} & 1 \\ {} & {} & {} & {} & 0 \end{pmatrix}
\right\}
\\ & \times
Q\left(\frac{r-1}{r},\frac{r-2}{r},\ldots,\frac{1}{r},\frac{0}{r}\right)
= n r {\rm I}_r - \begin{pmatrix} r-1 \\ {} & r-2 \\ {} & {} & \ddots \\ {} & {} & {} & 1 \\ {} & {} & {} & {} & 0 \end{pmatrix}
.
\end{split}
\end{equation*}
The $n$-part and the $1$-part of this equality follow from Lemma \ref{lemmaUsingRes} (i) and (ii) respectively applying (\ref{PU}), (\ref{Smult}), (\ref{S1}), (\ref{SU}) with $m=r$.
\end{proof}
%
%
\subsection{Proof of Theorem 5}
We take $r \ge 2$
and assume $s_1$, \ldots, $s_r$, $n$ are independent variables over $\Q$.
We put $\zeta_i = \dfrac{n}{rn - (r-i)}$ ($1 \le i \le r$).
Note that $\zeta_r=1/r$.
We define the following ring.
\[
{\mathscr R} = \Q[\zeta_1, \ldots, \zeta_{r-1}] \subset \Q(n)
\]
We have relations for $\zeta_i$.
\[
\zeta_i \zeta_j = \frac{1}{r} \frac{r-i}{j-i} \zeta_i + \frac{1}{r} \frac{r-j}{i-j} \zeta_j \ \ \ (\ 1 \le i,\ j\le r-1,\ i \ne j\ )
\]
Hence, we have
\[
{\mathscr R} = \Q \oplus \bigoplus_{1 \le i \le r-1} \bigoplus_{k \ge 1}\Q \zeta_i^k.
\]
Let ${\mathscr I}$ be the ideal of ${\mathscr R}$ given by
\[
{\mathscr I} = (\zeta_1) = \cdots = (\zeta_{r-1}) = (\zeta_1, \ldots, \zeta_{r-1}).
\]
Then, we have
\(
{\mathscr R}/{\mathscr I} = \Q.
\)
We put
\[
R^u = R
\begin{pmatrix}
rn-(r-1) \\
{} & rn-(r-2) \\
{} & {} & \ddots \\
{} & {} & {} & rn -0
\end{pmatrix}^{-1}.
\]
Then, the components of $R^u$ are contained in ${\mathscr R}[s_1, \ldots, s_r]$.
\[
(R^u)_{i,j} =  \{ 1 - (i-j) \zeta_j \} s_{i-j}  \ \ \ (1\le i\le r+d,\ 1\le j\le r)
\]
In particular, we have $(R^u)_{j,j}=1$ for $1 \le j \le r$.
Let $R_1^u$ and $R_2^u$ be the upper $r \times r$ matrix and the lower $d \times r$ matrix of $R^u$ respectively.
\begin{equation} \label{RuRu1Ru2}
R^u = \begin{pmatrix} R_1^u \\ R_2^u \end{pmatrix}
\end{equation}
Note that $R_2 R_1^{-1} = R_2^u R_1^{u-1}$.
Let $R^{um}$, $R_1^{um}$, $R_2^{um}$ be the matrices obtained by reducing $R^{u}$, $R_1^u$, $R_2^u$ modulo ${\mathscr I}$ respectively.
\begin{equation} \label{Rum}
R^{um} = 
\begin{pmatrix}
R_1^{um} \\ 
R_2^{rm}
\end{pmatrix}
=
\begin{pmatrix}
s_0 \\
s_1     & s_0 \\
\vdots  & \vdots  & \ddots \\
\vdots  & \vdots  & {}     & s_0 \\
s_{r-1} & s_{r-2} & \cdots & s_1 & \frac{r}{r}s_0   \\ 
s_r     & s_{r-1} & \cdots & s_2 & \frac{r-1}{r}s_1 \\
{}      & s_r     & {}           & \vdots           & \vdots \\
{}      & {}      & \ddots       & s_{r-1}          & \vdots \\
{}      & {}      & {}           & s_r              & \frac{1}{r}s_{r-1} 
\end{pmatrix}
\end{equation}
%
%
\begin{lemma}\label{lemmaR2umR1uminv}
\begin{enumerate}
\item
\ \ \ 
$
\begin{pmatrix}
- R_2^{um} R_1^{um-1} & {\rm I}_d
\end{pmatrix}
V = 0
$
\item
The rightmost column of $R_2 R_1^{-1} - R_2^{um} R_1^{um-1}$ is $0$.
\end{enumerate}
\end{lemma}
\begin{proof}
\noindent
(i)
Let $V_1$ and $V_2$ be the upper $r \times d$ matrix and the lower $d \times d$ matrix in $V$ respectively.
In general, for a matrix A, the matrix obtained by deleting the rightmost column of $A$ is denoted by $\langle A \rangle$.
Then, for matrices $A$ and $B$ having product $AB$, we have $\langle A B \rangle = A \langle B \rangle$. 
Then, we obtain
\begin{align*}
\begin{pmatrix} - R_2^{um} R_1^{um-1} & {\rm I}_d \end{pmatrix} V 
= &
- R_2^{um} R_1^{um-1} V_1 + V_2
\\
= &
- R_2^{um} S_r(\varphi(t))^{-1} \langle S_r(\varphi(t)) \rangle + V_2
\\
= &
- \langle R_2^{um} S_r(\varphi(t))^{-1} S_r(\varphi(t)) \rangle + V_2
\\
= &
- \langle R_2^{um} \rangle + V_2
\\
= &
- \begin{pmatrix} s_r & \cdots & s_2 \\ {} & \ddots & \vdots \\ {} & {} & s_r \end{pmatrix}
+ \begin{pmatrix} s_r & \cdots & s_2 \\ {} & \ddots & \vdots \\ {} & {} & s_r \end{pmatrix}
\\
= &
\ 0.
\end{align*}
\noindent
(ii)
Both of the rightmost columns of $R_2^u {R_1^u}^{-1}$ and $R_2^{um} R_1^{um-1}$ are 
\[
\begin{pmatrix}
\dfrac{r-1}{r} s_1 \\
\dfrac{r-2}{r} s_2 \\
\vdots \\
\dfrac{1}{r} s_{r-1}
\end{pmatrix}.
\]
\end{proof}
%

%
By Lemma \ref{lemmaR2umR1uminv},
it is sufficient to show
\[
\begin{pmatrix}
-R_2 R_1^{-1} + R_2^{um} R_1^{um-1} & {\rm O}_d
\end{pmatrix}
V = B_r \ Q_r \ Z_{r,n} \ P_r,
\]
i.e.,
\begin{equation}\label{target2}
(-R_2 R_1^{-1} + R_2^{um} R_1^{um-1}) \ V_1 = B_r \ Q_r \ Z_{r,n} \ P_r,
\end{equation}
to show (\ref{target1}).
Here, ${\rm O}_d$ denotes the $d \times d$ zero matrix.
In general, for $d\times d$ matrices $A$ and $B$, we have
\[
\begin{pmatrix}
\multicolumn{3}{c}{\multirow{3}{*}{$A$}}                  & 0      \\
{}               & {}                  & {}               & \vdots \\
{\phantom{\ast}} & {\phantom{\cdots}}  & {\phantom{\ast}} &  0      
\end{pmatrix}
\begin{pmatrix}
\multicolumn{3}{c}{\multirow{3}{*}{$B$}} & 0      \\
{}   & {}     & {}                       & \vdots \\
{}   & {}     & {}                       & 0      \\
\ast & \cdots & \ast                     & \ast
\end{pmatrix}
=
\begin{pmatrix}
\multicolumn{3}{c}{\multirow{3}{*}{$AB$}}                 & 0      \\
{}               & {}                  & {}               & \vdots \\
{\phantom{\ast}} & {\phantom{\cdots}}  & {\phantom{\ast}} &  0      
\end{pmatrix}.
\]
This implies that the rightmost column of $-R_2 + R_2^{um} R_1^{um-1} R_1$ is zero and, by Lemma \ref{lemmaR1inv}, we obtain the following.
\begin{align*}
&
(-R_2 R_1^{-1} + R_2^{um} R_1^{um-1}) V_1 
\\
=
&
(-R_2+R_2^{um} R_1^{rm-1} R_1) \; \frac{1}{n}\;  
Q\left(\frac{r-1}{r},\frac{r-2}{r},\cdots,\frac{0}{r}\right) 
\begin{pmatrix} \zeta_1 \\ {} & \zeta_2 \\ {} & {} & \ddots \\ {} & {} & {} & \zeta_r \end{pmatrix}
\\
& \times
P\left(-\frac{2r-1}{r},-\frac{2r-2}{r},\cdots,-\frac{r}{r}\right)
V_1
\\
=
&
\langle -R_2+R_2^{um} R_1^{rm-1} R_1\rangle \;\frac{1}{n}\; 
Q\left(\frac{r-1}{r},\frac{r-2}{r},\cdots,\frac{1}{r}\right) 
\begin{pmatrix} 
\zeta_1 \\ 
{} & \zeta_2{} \\
{} & {} & \ddots \\
{} & {} & {} & \zeta_{r-1} 
\end{pmatrix}
\\
& \times
P\left(-\frac{2r-1}{r},\frac{2r-2}{r},\cdots,-\frac{r+1}{r}\right)
S_d (\varphi(t))
\end{align*}
Therefore, by (\ref{PU}), (\ref{UQ}), (\ref{SU}), 
it is sufficient to show
\begin{equation}\label{target3}
\langle ( -R_2 + R_2^{um} R_1^{um-1} R_1) S_r(\varphi(t)) \rangle = n \ B_r,
\end{equation}
to show (\ref{target2}).
It is immediate to verify the $1$-part of (\ref{target3}).
The $n$-part of (\ref{target3}) follows from the following Proposition.
%
%
\begin{prop}
\begin{align}
&-\nonumber 
\begin{pmatrix}
0 s_r & 1 s_{r-1} & \cdots & (r-2)s_2 \\
{}    & 0 s_r     & \ddots & \vdots   \\
{}    &           & \ddots & 1s_{r-1} \\
{}    &           &        & 0s_r
\end{pmatrix}
\begin{pmatrix}
s_0     \\
s_1     & s_0    \\
\vdots  & \ddots & \ddots \\
s_{r-2} & \cdots & s_1    & s_0
\end{pmatrix}
\\
&+ \label{formulaBr} 
\begin{pmatrix}
s_r & s_{r-1} & \cdots & s_2     \\
{}  & s_r     & \ddots & \vdots  \\
{}  & {}      & \ddots & s_{r-1} \\
{}  & {}      & {}     & s_r
\end{pmatrix}
\begin{pmatrix}
rs_0 \\
(r-1)s_1 & rs_0 \\
\vdots & \ddots & \ddots \\
2s_{r-2} & \cdots & (r-1)s_1 & rs_0
\end{pmatrix}
\\
&-\frac{1}{r} \nonumber
\begin{pmatrix}
(r-1)s_1 \\
(r-2)s_2 \\
\vdots \\
1 s_{r-1}
\end{pmatrix}
\begin{pmatrix}
(r-1)s_{r-1} & (r-2)s_{r-2} & \ldots & 1s_1
\end{pmatrix}
= 
B_r
\end{align}
\end{prop}
%
%
%
\begin{proof}
%
%
We denote the left hand side of (\ref{formulaBr}) by $LHS$.
We also denote the first , the second and the third term of $LHS$ by $LHS_1$, $LHS_2$ and $LHS_3$ respectively.
By (\ref{BezDef}) and (\ref{matBr}), we shall show
\begin{equation*}
%
\begin{pmatrix}
x^{d-1} & x^{d-2} & \cdots & 1
\end{pmatrix}
{LHS}
%
\begin{pmatrix}
1 \\
y \\
\vdots \\
y^{d-1}
\end{pmatrix}
(x-y)
= f'(x) \left\{ f(y) - \frac{1}{r} y f'(y) \right\} - f'(y) \left\{ f(x) - \frac{1}{r} x f'(x) \right\}.
\end{equation*}
%
%
We remark that, in general, for a polynomial $\Phi(x,\eta)$ of degree $m$ of the form
\[
\Phi(x,\eta) = \sum_{0 \le i \le m,\, 0 \le j \le m} c_{i,j} x^i \eta^j,
\]
we have
\begin{align}
& \nonumber
\left\{ \sum_{0\le k \le m} \left( \sum_{0\le i \le k,\, 0 \le j \le k} c_{i,j} x^i \eta^j \right) (x\eta)^{m-k} \right\} (x\eta - 1)
\\
& = \nonumber
- \Phi(x,\eta)
+ \sum_{0\le i \le m,\, 0\le j \le m,\, \max(i,j)=m} \left( \sum_{0\le k \le \min(i,j)} c_{i-k,j-k} \right) x^{i+1} \eta^{j+1}.
\end{align}
%
%
By this equality, putting $\eta = 1/y$, we obtain
\begin{align}
& \nonumber
\begin{pmatrix}
x^{d-1} & x^{d-2} & \cdots & 1
\end{pmatrix}
{LHS_1}
\begin{pmatrix}
1 \\
y \\
\vdots \\
y^{d-1}
\end{pmatrix}
(x-y)
\\
&= \nonumber
-y^d
\left\{ \sum_{0 \le k \le d-1}  \left( \sum_{0 \le i \le k} i s_{r-i} x^i \right) \left( \sum_{0 \le j \le k} s_j \eta^j \right) (x\eta)^{d-1-k} \right\} (x\eta-1)
\\
&= \nonumber
\left(x f'(x) - r s_0 x^r - (r-1) s_1 x^{r-1}\right) \frac{1}{y} \left( f(y) - s_{r-1} y - s_r \right) 
\\
&\phantom{=}+ \nonumber
\sum_{0 \le i \le d-1,\, 0 \le j \le d-1,\, \max(i,j)=d-1} \left( \sum_{0 \le k \le \min(i,j)} (-i+k) s_{r-i+k} s_{j-k} \right) x^{i+1} y^{r-2-j} .
\end{align}
%
%
Similarly, we obtain
\begin{align}
& \nonumber
\begin{pmatrix}
x^{d-1} & x^{d-2} & \cdots & 1
\end{pmatrix}
{LHS_2}
\begin{pmatrix}
1 \\
y \\
\vdots \\
y^{d-1}
\end{pmatrix}
(x-y)
\\
&= \nonumber
y^d
\left\{ \sum_{0 \le k \le d-1}  \left( \sum_{0 \le i \le k} s_{r-i} x^i \right) \left( \sum_{0 \le j \le k} (r-j) s_j \eta^j \right) (x\eta)^{d-1-k} \right\} (x\eta-1)
\\
&= \nonumber
- \left(f(x) - s_0 x^r - s_1 x^{r-1} \right) \left( f'(y) - s_{r-1} \right) 
\\
&\phantom{=}+ \nonumber
\sum_{0 \le i \le d-1,\, 0 \le j \le d-1,\, \max(i,j)=d-1} \left( \sum_{0 \le k \le \min(i,j)} (r-j+k) s_{r-i+k} s_{j-k} \right) x^{i+1} y^{r-2-j} .
\end{align}
%
%
We have the sum
\[
\sum_{0 \le i \le d-1,\, 0 \le j \le d-1,\, \max(i,j)=d-1} \left( \sum_{0 \le k \le \min(i,j)} (r-i-j+2k) s_{r-i+k} s_{j-k} \right) x^{i+1} y^{r-2-j},
\]
of the two second terms in the right hand sides of the above two equalities.
For indices $k$ and $k'$ such that $k+k'=i+j-r$, we have the cancellation
\[
(r-i-j+2k) s_{r-i+k} s_{j-k} + (r-i-j+2k') s_{r-i+k'} s_{j-k'} = 0.
\]
By this cancellation, this sum is equal to
\begin{align}
& \nonumber
s_{r-1} \left( x f'(x) - f(x) - (r-1) s_0 x^r + s_r \right) 
+ s_r \left( f'(x) - s_{r-1} \right)
\\ 
& \nonumber
+ s_1 x^{r-1} \left( (r-1) \frac{1}{y} ( f(y) - s_r ) - f'(y) + s_0 y^{r-1} \right)
+ s_0 x^{r-1} \left( r f(y) - y f'(y) - s_1 y^{r-1} \right)
\\
& \nonumber
- (r-2) s_1 s_{r-1} x^{r-1} - r s_0 s_r x^{r-1}.
\end{align}
%
%
We also have
\begin{align}
& \nonumber
\begin{pmatrix}
x^{d-1} & x^{d-2} & \cdots & 1
\end{pmatrix}
{LHS_3}
\begin{pmatrix}
1 \\
y \\
\vdots \\
y^{d-1}
\end{pmatrix}
(x-y)
\\
&= \nonumber
- \frac{1}{r} \left( f'(x) - r s_0 x^{r-1} \right) \frac{1}{y} \left( r f(y) - y f'(y) - r s_r \right) (x-y).
\end{align}
By all of these computations, we obtain the assertion.
\end{proof}
This completes the proof of Theorem 5.
%
%
%
\section{Two Similar Experimental Equalities}
In this Section, we shall give two experimental equalities.
These two have some similarities.
%
%
\subsection{The first experimental equality}
Let $p$ be a prime.
We put
\begin{equation*}
{\mathscr E} (p) = \left\{ 
(r,e,d) \in \Z^3 \left|
\begin{array}{l}
r \ge 2,\ 0 \le e \le p-1,\ 0 \le d \le r-1,\ d \le p,\ r-1-d \le p, \\
d(p-1) \le re \le (d+1)(p-1)
\end{array}
\right.
\right\}.
\end{equation*}
Then, the set of $(r,e)$ for $(r,e,d)\in{\mathscr E}(p)$ is illustrated as follows.
\begin{center}
\includegraphics{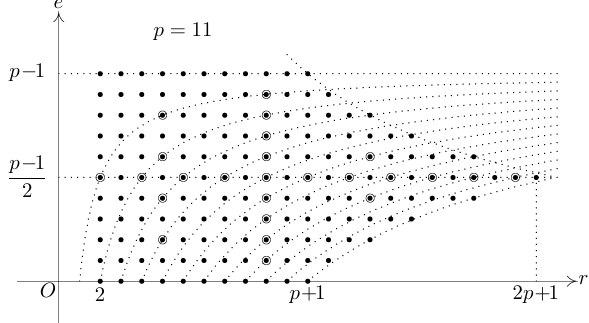}
\captionof{figure}{The projection of ${\mathscr E}(p)$ for $p=11$}
\end{center}
Here, points $(r,e)$ with two $(r,e,d)\in{\mathscr E}(p)$ are shown as circled dots.
For $(r,e,d)\in{\mathscr E}(p)$, we take a generic polynomial $f(x)=s_0 x^r + s_1 x^{r-1} + \cdots + s_r$ 
of degree $r$ in characteristic $p$ and define the matrix $M_d(f(x)^e)$ by (\ref{defMd}).
For $(r,e,d)\in{\mathscr E}(p)$, we have
\begin{flalign*}
(r-1,e,d-1) \in {\mathscr E}(p) & \Longleftrightarrow  r \ge 3,\ d(p-1)-(r-1)e \ge 0,\ d \ge 1, \\
(r-1,e,d) \in {\mathscr E}(p) & \Longleftrightarrow  r \ge 3,\ d(p-1)-(r-1)e \le 0,\ d \le r-2 . 
\end{flalign*}
These two cases are not exclusive to each other.
When $(r-1,e,d-1) \in {\mathscr E}(p)$, the last row of $M_d(f(x)^e)$ is divisible by $s_0^{d(p-1)-(r-1)e}$, and we have
\begin{equation}\label{recursive1}
\left. \frac{\det M_d(f(x)^e)}{s_0^{d(p-1)-(r-1)e}} \right|_{s_0 = 0} 
=
(-1)^{d-1} \binom{e}{re-d(p-1)} s_1^{re-d(p-1)} \det M_{d-1}(f_1(x)^e).
\end{equation}
Here, we put $f_1(x)=s_1 x^{r-1} + \cdots + s_r$.
When $(r-1,e,d)\in{\mathscr E}(p)$, we have
\begin{equation}\label{recursive2}
\left. \det M_d(f(x)^e) \right|_{s_0 = 0} = \det M_d(f_1(x)^e).
\end{equation}
For the discriminant $\Delta(f(x))$ of $f(x)$,
we remark that
\begin{equation*}
\left. \Delta(f(x))\right|_{s_0=0} = s_1^2 \Delta(f_1(x)).
\end{equation*}
\begin{lemma}
\phantom{This is a phantom line}
\begin{enumerate}
\item For $(r,e,d)\in{\mathscr E}(p)$, we have $\det M_d(f(x)^e) \ne 0$.
\item For $(r,e,d)\in{\mathscr E}(p)$, $\det M_d(f(x)^e)$ is divisible exactly by $s_0^{\max(0,d(p-1)-(r-1)e)}$.  
\item For $(r,e,d)\in{\mathscr E}(p)$, $\det M_d(f(x)^e)$ is divisible exactly by $\Delta(f(x))^{\max(0,e-(p-1)/2)}$.
\end{enumerate}
\end{lemma}
\begin{proof}
When $e-(p-1)/2>0$, $\det M_d(f(x)^e)$ is divisible at least by $\Delta(f(x))^{e-(p-1)/2}$ (Lemma \ref{lemma2}).
The assertions (i) and (iii) for $r=2$ are obtained in Theorem 1.
For $r \ge 3$, we apply (\ref{recursive1}) and (\ref{recursive2}) inductively on $r$ to obtain (i) and (iii). 
\end{proof}

For $(r,e,d) \in {\mathscr E}(p)$, we put $\ehat = (p-1)-e$ and $\dhat = (r-1)-d$.
Then, 
\begin{equation*}
{\mathscr E}(p) \ni (r,e,d) \longmapsto (r,\ehat,\dhat) \in {\mathscr E}(p)
\end{equation*}
is an involutive bijection.

The first experimental equality is that, for $(r,e,d)\in{\mathscr E}(p)$, 
\begin{align}
&\frac{ \det M_d(f(x)^e ) } { \det M_{\widehat{d}}(f(x)^{\widehat{e}}) }
=
\varepsilon_{p,r,e,d}\, s_0^{d(p-1)-(r-1)e}\,  \Delta(f(x))^{e-(p-1)/2}, \label{ExperimentalEquality1}\\
&\varepsilon_{p,r,e,d}
=
(-1)^{\frac{r(r+1)}{2} (1+e) + (r+1) d } \,\{ (d+1)(p-1)-re\}!\, e!^{r}. \notag
\end{align}
%
%
We have the following properties, which fit into the equality (\ref{ExperimentalEquality1}).
\begin{align*}
&
\varepsilon_{p,r,e,d} \, \varepsilon_{p,r,\ehat,\dhat} = 1 \\
&
s_0^{d(p-1)-(r-1)e} \, s_0^{\dhat(p-1)-(r-1)\ehat} = 1 \\
&
\Delta(f(x))^{e-(p-1)/2} \, \Delta(f(x))^{\ehat-(p-1)/2} = 1 
\end{align*}
%
%
The equality of degrees of the both sides of (\ref{ExperimentalEquality1}) is as follows.
For the degree such that $\deg(s_i)=i$, we have
\begin{equation*}
\left\{ red - \frac{d(d+1)}{2} (p-1)  \right\} - 
\left\{ r\ehat\dhat - \frac{\dhat(\dhat+1)}{2} (p-1) \right\} 
= r(r-1) \left(e-\frac{p-1}{2}\right),
\end{equation*}
and for the degree such that $\deg(s_i)=1$, we have
\begin{equation*}
ed - \ehat\dhat = \{ d(p-1) - (r-1)e \} + 2(r-1) \left(e-\frac{p-1}{2}\right).
\end{equation*}

Since we have 
\begin{equation*}
{\mathscr D}(p) \setminus {\mathscr D}_0 (p) = \{ (r,e,d) \in {\mathscr E}(p)| e-(p-1)/2 > 0 \},
\end{equation*}
Theorem 1 is a special case of (\ref{ExperimentalEquality1}).
Also, when $e-(p-1)/2=0$, the equality (\ref{ExperimentalEquality1}) can be easily verified.

The equality (\ref{ExperimentalEquality1}) remains unproven.
%
%
\subsection{The second experimental equality}
In \cite{Glynn}, Glynn proved, among others, the following theorem, for which we give an elementary proof.

%
%
\begin{thrm}[Glynn \cite{Glynn}, Theorem 4.1]
Let $p$ be a prime and $A=\left( a_{i,j} \right)_{1 \le i,j \le r}$ a square matrix in characteristic $p$.
Then, 
\begin{equation}\label{GlynnEqualityCoeff}
\left[ \prod_{1 \le j \le r} X_j^{p-1}\right] \left\{ \prod_{1 \le i \le r} \left( \sum_{1 \le j \le r} a_{i,j} X_j \right)^{p-1}\right\} = \det A^{\,p-1}
\end{equation}
holds.
Here, $\left[ \prod_{1 \le j \le r} X_j^{p-1}\right]$ means the coefficient of $\, \prod_{1 \le j \le r} X_j^{p-1}$.    
\end{thrm}

\begin{proof}
The equality (\ref{GlynnEqualityCoeff}) is equivalent to
\begin{equation}\label{GlynnEqualityDiff}
(-1)^r  \prod_{1 \le j \le r} \left( \dfrac{\partial}{\partial X_j}\right)^{p-1}\prod_{1 \le i \le r} \left( \sum_{1 \le j \le r} a_{i,j} X_j \right)^{p-1} = \det A^{\,p-1},
\end{equation}
which we shall prove by induction on $r$.
The case $r=1$ is immediate.
We assume $r > 1$.
For $1 \le i \le r$, we put $L_i = \sum_{2 \le j \le r} a_{i,j} X_j$.
We have
\begin{align*}
&
\left( \frac{\partial}{\partial X_1}\right)^{p-1} \prod_{1 \le i \le r} \left( a_{i,1}X_1+L_i\right)^{p-1} \\
&=
\prod_{1 \le i \le r} \left( a_{i,1}X_1+L_i\right)^{p} \left( \frac{\partial}{\partial X_1}\right)^{p-1} \prod_{1 \le i \le r}\frac{1}{a_{i,1}X_1+L_i} \\
&=
\prod_{1 \le i \le r} \left( a_{i,1}X_1+L_i\right)^{p} \left( \frac{\partial}{\partial X_1}\right)^{p-1} 
\sum_{1 \le i \le r} 
\left\{ \prod_{\substack{1 \le m \le r \\ m\ne i}} \frac{1}{L_m - \frac{a_{m,1}}{a_{i,1}} L_i}\right\} \frac{1}{ a_{i,1}X_1+L_i} \\
&=
\sum_{1 \le i \le r} \frac
{\prod_{\substack{1 \le m \le r \\ m\ne i}}(a_{m,1}X_1+L_m)^p }
{\prod_{\substack{1 \le m \le r \\ m\ne i}} (L_m -\frac{a_{m,1}}{a_{i,1}} L_i)}
(- a_{i,1}^{p-1}).
\end{align*}
For $i$ ($1 \le i \le r$), we let $A_i$ be the following matrix.
For each $m$ ($1 \le m \le r$, $m\ne i$), we subtract the $i$-th row of $A$ multiplied by $a_{m,1}/a_{i,1}$ from the $m$-th row of $A$.
Furthermore, we delete the first column and the $i$-th row.
The resulting matrix is denoted by $A_i$.
We have $(-1)^{i-1} a_{i,1} \det A_i = \det A$.
For a fixed $i$ ($1 \le i \le r$), we can apply the induction assumption as follows. 
\begin{align*}
&
\prod_{2 \le j \le r} \left( \frac{\partial}{\partial X_j}\right)^{p-1} \frac{1}{\prod_{\substack{1 \le m \le r \\ m \ne i}} ( L_m - \frac{a_{m,1}}{a_{i,1}} L_i)} \\
&=
\frac{1}{\prod_{\substack{1 \le m \le r \\ m \ne i}} ( L_m - \frac{a_{m,1}}{a_{i,1}} L_i)^p}  
\prod_{2 \le j \le r} \left( \frac{\partial}{\partial X_j}\right)^{p-1}  \prod_{\substack{1 \le m \le r \\ m \ne i}} ( L_m - \frac{a_{m,1}}{a_{i,1}} L_i)^{p-1} \\
&=
\frac{1}{\prod_{\substack{1 \le m \le r \\ m \ne i}} ( L_m - \frac{a_{m,1}}{a_{i,1}} L_i)^p}  
(-1)^{r-1} \det A_i^{p-1}
\end{align*}
Then, the left hand side of (\ref{GlynnEqualityDiff}) is equal to
\begin{equation*}
\sum_{1 \le i \le r}
\left\{ \prod_{\substack{1 \le m \le r \\ m\ne i}} (X_1 + \frac{L_m}{a_{m,1}}) \right\}^p
\left\{ \prod_{\substack{1 \le m \le r \\ m \ne i}} \frac{1}{(X_1+\frac{L_m}{a_{m,1}} )-(X_1 + \frac{L_i}{a_{i,1}} )}\right\}^p \det A^{\, p-1}.
\end{equation*}
We apply the quality
\begin{equation*}
\sum_{1 \le i \le r} 
\left( \prod_{\substack{1 \le m \le r \\ m \ne i}} u_m\right)
\left( \prod_{\substack{1 \le m \le r \\ m \ne i}} \frac{1}{u_m - u_i}\right) = 1,
\end{equation*}
which is obtained by putting $u=0$ on the partial fraction decomposition of $\dfrac{1}{\prod_{1 \le i \le r} (u-u_i)}$,
to obtain (\ref{GlynnEqualityDiff}).
\end{proof}
Now we give the second experimental equality.
Let $A = \left( a_{i,j} \right)_{1 \le i,j \le r}$ be a square matrix of size $r \ge 1$ in characteristic $p$.
For $0 \le e \le p-1$, we put
\begin{equation*}
G^e(A) = \left[ \prod _{1 \le j \le r} X_j^e \right] \prod_{1 \le i \le r} \left( \sum_{1 \le j \le r} a_{i,j} X_j \right)^e.
\end{equation*}
Here, $\left[ \prod_{1 \le j \le r} X_j^e \right]$ means the coefficient of $\prod_{1 \le j \le r} X_j^e$.

For $0 \le e \le p-1$, we put $\widehat{e} = (p-1)-e$.
We denote the adjoint matrix of $A$ by $\widehat{A}$, i.e., $ A \widehat{A} = \det A$.
Then, the second experimental equality is that
\begin{equation}\label{ExperimentalEquality2}
\frac{G^e (A)}{G^{\widehat{e}}(\widehat{A})  } = \det A^{\, p-1-r \widehat{e} }.
\end{equation}
The equality (\ref{ExperimentalEquality2}) for $0 < e < p-1$ remains unproven.
%
%
%
%
%
%
\bibliographystyle{plain}
\bibliography{MyBib}
\end{document}